\documentclass[a4paper]{amsart}

\usepackage[utf8]{inputenc}
\usepackage{amsthm,amssymb,amsmath,mathtools,mathpazo,xargs,enumerate,graphicx,subcaption,wasysym}
\usepackage[linecolor=orange,backgroundcolor=orange!25,bordercolor=orange]{todonotes}
\usepackage[colorlinks]{hyperref}
\usepackage{mdframed}

\title[Asymptotic orthogonality in hyperbolic groups]{Asymptotic Schur orthogonality in hyperbolic groups with application to monotony}

\author{Adrien Boyer}
\address[A. Boyer]{Weizmann Institute of Science\\ Herzl St 234\\ Rehovot, 7610001, Israel}
\email{aadrien.boyer@gmail.com}

\author{{\L}ukasz Garncarek}
\address[{\L}. Garncarek]{Weizmann Institute of Science\\ Herzl St 234\\ Rehovot, 7610001, Israel}
\email{lukgar@gmail.com} 


\thanks{Adrien Boyer was partially supported by ERC Grant 306706. Łukasz Garncarek was partially supported by Narodowe Centrum Nauki grant 2012/06/A/ST1/00259.}

\newtheorem{theorem}{Theorem}[section]
\newtheorem*{conjecture}{Conjecture}
\newtheorem*{theorem*}{Theorem}
\newtheorem{lemma}[theorem]{Lemma}

\newtheorem{proposition}[theorem]{Proposition}
\newtheorem{corollary}[theorem]{Corollary}
\theoremstyle{definition}

\theoremstyle{remark}
\newtheorem{remark}[theorem]{Remark}
\numberwithin{equation}{section}

\newcommand{\RR}{\mathbb{R}}

\newcommand{\CC}{\mathbb{C}}
\newcommand{\one}{\boldsymbol{1}}
\newcommand{\id}{\mathrm{id}}

\DeclareMathOperator{\diam}{diam}
\DeclareMathOperator{\Lip}{Lip}
\DeclareMathOperator{\Prob}{Prob}

\newcommand{\card}[1]{\left\lvert #1 \right\rvert}
\newcommand{\norm}[1]{\left\lVert#1\right\rVert}
\newcommand{\abs}[1]{\left\lvert#1\right\rvert}

\newcommand{\conj}{\overline}
\newcommand{\lasymp}{\prec}
\newcommand{\rasymp}{\succ}
\newcommand{\aless}{\apprle}
\newcommand{\agtr}{\apprge}
\newcommand{\res}{\vert}

\newcommand{\bd}{\partial}
\newcommand{\ov}{\overline}

\renewcommand{\tilde}{\widetilde}

\begin{document}
\maketitle

\begin{abstract}
  We prove a generalization of Schur orthogonality relations for certain classes of representations of Gromov hyperbolic groups. We apply the obtained results to show that representations of non-abelian free groups associated to the Patterson-Sullivan measures corresponding to a wide class of invariant metrics on the group are monotonous in the sense introduced by Kuhn and Steger. This in particular includes representations associated to harmonic measures of a wide class of random walks.
\end{abstract}

\section{Introduction} 
\label{sec:introduction}

Given an irreducible unitary representation \(\pi\) of a group \(\Gamma\) on a Hilbert space \(\mathcal{H}\), one may consider its \emph{matrix coefficients}, i.e.\ the functions on the group of the form \(g\mapsto \langle \pi(g)v,w \rangle\) with \(v,w\in\mathcal{H}\). If \(\Gamma\) is a finite group, then \(\mathcal{H}\) is finite-dimensional, and the Schur orthogonality relations tell us that
\begin{equation*}
  \frac{\dim\mathcal{H}}{\abs{\Gamma}} \sum_{g\in \Gamma} \langle \pi(g) v_{1},w_{1}\rangle\conj{\langle \pi(g)v_{2},w_{2}\rangle} =  \langle v_{1},v_{2}\rangle \conj{\langle w_{1},w_{2}\rangle}.
\end{equation*}
In this form, the ortogonality relations do not make sense for infinite groups and infinite-dimensional representations. However, the sum in the left-hand side of the above equation could be replaced by some kind of average over either a ball, or an annulus, and we may consider the limit of the obtained expression as the radius goes to infinity.

The aim of the paper is to make that statement precise for a certain class of representations of Gromov-hyperbolic groups, studied in \cite{Garncarek2014}, generalizing the representations of fundamental groups of negatively curved manifolds introduced in \cite{BaderMuchnik2011}. If \(\Gamma\) is a non-elementary Gromov hyperbolic group endowed with an invariant metric \(d\) satisfying some regularity assumptions, the Gromov boundary \(\bd\Gamma\) of \(\Gamma\) can be endowed with a \emph{Patterson-Sullivan measure} \(\mu_{PS}\) corresponding to \(d\). The class of \(\mu_{PS}\) is invariant under the action of \(\Gamma\), yielding a quasi-regular representation \(\pi\) on \(L^{2}(\bd\Gamma,\mu_{PS})\), given by
\begin{equation*}
  [\pi(g)v](\xi) = \sqrt{\frac{dg_{*}\mu}{d\mu}(\xi)} v(g^{-1}\xi),
\end{equation*}
which we call the \emph{boundary representation} associated to the metric \(d\). The main result of our paper, the Asymptotic Orthogonality Theorem, reads as follows.

\begin{theorem*}
Let \(\pi\) be a boundary representation of a hyperbolic group \(\Gamma\) corresponding to the metric \(d\) with Patterson-Sullivan measure \(\mu_{PS}\). Let \(A_{R}=\{g\in\Gamma : d(1,g) \in [R-h,R+h]\}\), where \(h\) is sufficiently large. Then there exists a one-parameter family of probability measures \(\mu_{R}\) supported in \(A_{R}\), satisfying \(\mu_{R}(g)\leq C /\abs{A_{R}}\) for some \(C>0\) independent of \(R\), such that for any \(v_{1},v_{2},w_{1},w_{2} \in L^{2}(\bd\Gamma, \mu_{PS})\) and \(f_{1},f_{2}\in C(\Gamma \cup \bd\Gamma)\)
\begin{multline*}
  \lim_{R\to\infty} \sum_{g\in A_{R}} \frac{\mu_{R}(g)f_{1}(g)f_{2}(g^{-1})}{\langle  \pi(g) \one_{\bd\Gamma},\one_{\bd\Gamma}\rangle^{2}}\langle \pi(g)v_{1},w_{1} \rangle \conj{\langle \pi(g)v_{2},w_{2}\rangle} = \\ = \langle f_{2}\res_{\bd\Gamma} v_{1},v_{2}\rangle\conj{\langle w_{1},f_{1}\res_{\bd\Gamma}w_{2}\rangle}.
\end{multline*}
\end{theorem*}

In the case of a convex cocompact group of isometries of a CAT(-1) space with a non-arithmetic spectrum, we refine this result further, obtaining \(\mu_{R}(g) = 1/\abs{A_{R}}\). Moreover, apart from Patterson-Sullivan measures, there is another important family of measures, defined on the boundaries of fundamental groups of negatively curved manifolds, called the \emph{Gibbs streams}. We prove a variant of the Asymptotic Orthogonality Theorem also for the quasi-regular representations associated with these measures, studied in \cite{BoyerMayeda2016}. Finally, by \cite{Blachere2011}, the harmonic measures of a large family of random walks are in fact Patterson-Sullivan measures corresponding to metrics satisfying our assumptions, and therefore the Asymptotic Orthogonality Theorem holds also for the quasi-regular representations corresponding to these harmonic measures.

As an application of the Asymptotic Orthogonality Theorem, we extend the results of Kuhn and Steger about \emph{monotony} of free group representations. In \cite{Kuhn2001}, they propose a method of classifying the unitary representations of a free group \(\Gamma\), which are weakly contained in the regular representation, using the notion of \emph{boundary realization}. Roughly speaking, a boundary realization of a representation \(\pi\) on a Hilbert space \(\mathcal{H}\) is a \(\Gamma\)-equivariant isometric embedding of \(\mathcal{H}\) into a space of the form \(L^{2}(\bd\Gamma,\mu)\), acted upon by the quasi-regular representation \(\pi_{\mu}\) corresponding to some measure \(\mu\), in such a way that the image is a cyclic subspace for the multiplication representation of \(C(\bd\Gamma)\). If this embedding is surjective, the realization is said to be \emph{perfect}.

If a representation \(\pi\) of a free group is weakly contained in the regular representation, then it admits a boundary realization, and Kuhn and Steger distinguish three different behaviors of the class of boundary realizations of \(\pi\), inspired by the case of restrictions of irreducible unitary representations of \(SL_{2}(\RR)\) to a free lattice \cite[Section 5]{Kuhn2001}. They say that \(\pi\) is \emph{monotonous} if it admits a unique boundary realization which is perfect, \emph{odd} if it admits a unique boundary realization which is imperfect, and \emph{duplicitous} if it admits exactly two perfect boundary realizations, and any imperfect realization is a combination of these two, in a certain precise sense. Then they conjecture the following.


\begin{conjecture}
  A unitary representation of a free group, weakly contained in the regular representation, is either monotonous, duplicitous, or odd.
\end{conjecture}

To motivate the conjecture, they prove that a certain class of  representations of free groups, associated to some explicitly constructed measures on their boundaries, are monotonous. As free groups are hyperbolic, the construction of \cite{Garncarek2014} provides them with a supply of boundary representations, and we manage to show the following.

\begin{theorem*}
  Boundary representations of free groups are monotonous.
\end{theorem*}

Our techniques combine functional analytic tools, namely property RD \cite{Chatterji2016}, with equidistribution results in the spirit of Margulis \cite{Margulis2004} and Roblin \cite{Roblin2003}, which we prove in the general context of Gromov hyperbolic groups.

\subsection{Organization}
\label{sec:organization}

The paper is organized as follows. In Section~\ref{sec:geometric-setting} we discuss all the preliminaries, explain the necessary facts involving hyperbolic groups, their boundaries and Patterson-Sullivan measures. We also discuss a slightly stronger variant of property RD satisfied by hyperbolic groups. Section \ref{sec:equid-asympt-schurs} deals with the proof of the Asymptotic Orthogonality Theorem. Sections \ref{sec:conv-cocomp-groups} and \ref{sec:simply-conn-manif} deal with the variants of the main theorem for convex cocompact groups of isometries of CAT(-1) spaces, and representations associated to Gibbs streams, respectively. Finally, in Section~\ref{sec:application-monotony} we present the application of our results to the phenomenon of monotony of representations of free groups.

\subsection{Acknowledgments}
\label{sec:acknowledgments}

We would like to thank Uri Bader for numerous helpful discussions and Christophe Pittet for his interest to this work. We are grateful to Gabriella Kuhn and Tim Steger for several explanations concerning the concepts of monotony, duplicity and oddity. Finally, we would like to thank Nigel Higson for informing us about Kuhn and Steger's conjecture of classifications of irreducible unitary representations of the free group.  

\section{The setting}
\label{sec:geometric-setting}

We begin by briefly discussing the setting of hyperbolic groups, the Gromov boundary and the Patterson-Sullivan measure in the generality of roughly geodesic metrics. A more detailed explanation of this setting can be found in the article \cite{Garncarek2014}, where boundary representations of hyperbolic groups are defined, and their irreducibility and classification are studied.

\subsection{Conventions for writing estimates}
\label{sec:conv-writ-estim}

Throughout the article we will follow some conventions regarding various kinds of estimates, allowing us to hide the constants whose values are inessential. 

First of all, we will always consider a group \(\Gamma\) endowed with a metric \(d\) from a certain class, and we will treat both these objects as fixed. In particular, if we refer to some quantity as \emph{constant}, we still allow it to depend on the choice of \(\Gamma\) and \(d\)---but not on the other parameters, unless explicitly stated otherwise.

We will write
\begin{equation}
  a \aless_{x,y,\ldots} b
\end{equation}
to indicate that there exists \(C=C(x,y,\ldots)\), possibly dependent on the parameters \(x,y,\ldots\) (and also on \(\Gamma\) and \(d\), as already mentioned) such that
\begin{equation}
  a \leq b + C.
\end{equation}
If both \(a\aless_{x,y,\dots} b\) and \(a \agtr_{x,y,\dots} b\) hold, we write \(a \approx_{x,y,\dots} b\). Analogous notation will be used for multiplicative estimates; we write
\begin{equation}
  a \lasymp_{x,y,\dots} b
\end{equation}
if \(a\leq Cb\) for some \(C>0\), possibly depending on \(x,y,\dots\), and if both \(a\lasymp_{x,y,\ldots}b\), and \(a\rasymp_{x,y,\dots} b\) hold, then we use the symbol \(a \asymp_{x,y,\dots} b\).

\subsection{Hyperbolicity}
\label{sec:hyperbolic-groups}

A metric space \((X,d)\) is said to be \emph{Gromov hyperbolic}, or \emph{hyperbolic} for short, if for any \(x,y,z\in X\) and some/any\footnote{if the condition holds for some \(o\) and \(\delta\), then it holds for any \(o\) and \(2\delta\)} basepoint \(o\in X\) one has
\begin{equation}
  (x,y)_{o}\agtr \min\{ (x,z)_{o},(z,y)_{o}\},
\end{equation}
where \((x,y)_{o}\) stands for the \emph{Gromov product} of \(x\) and \(y\) with respect to \(o\), that is
\begin{equation}
  (x,y)_{o}=\frac{1}{2}(d(x,o)+d(y,o)-d(x,y)).
\end{equation}
If we restrict to the class of geodesic metric spaces, hyperbolicity becomes invariant under quasi-isometries, which is not the case for arbitrary metric spaces \cite{Blachere2011}. In particular, if \(\Gamma\) is a finitely generated group, all its Cayley graphs associated with finite generating sets are pairwise quasi-isometric geodesic metric spaces, so either none or all of them are hyperbolic. In the latter case we say that \(\Gamma\) is a \emph{hyperbolic group}. 

Let \(\Gamma\) be a hyperbolic group. As the basepoint we take \(1\in\Gamma\) and omit it from notation of Gromov product. We will also write \(\abs{g}\) for \(d(1,g)\). We denote by \(\mathcal{D}(\Gamma)\) the class of all metrics \(d\) on \(\Gamma\) satisfying the following three conditions:
\begin{enumerate}
\item \(d\) is quasi-isometric to a word metric \(d_{S}\) on \(\Gamma\), i.e.\ there exist \(C\geq0\) and \(L\geq 1\) such that
  \begin{equation}
    \frac{1}{L}d_{S}(g,h)-C \leq d(g,h) \leq Ld_{S}(g,h)+C,
  \end{equation}
\item \(d\) is left-invariant, i.e.\ \(d(kg,kh)=d(g,h)\),
\item \(d\) is hyperbolic.
\end{enumerate}
By the discussion in \cite[Section 3.1]{Garncarek2014}, for any \(d\in\mathcal{D}(\Gamma)\) the metric space \((\Gamma,d)\) is \emph{roughly geodesic}. This means that for any \(x,y\in \Gamma\) there exists a \emph{roughly geodesic segment}  joining \(x\) and \(y\), i.e.\ a map \(\gamma\colon [0,\ell]\to \Gamma\), such that
\begin{equation}\label{eq:rough-geo-uniform-est}
   d(\gamma(s),\gamma(t)) \approx \abs{s-t},\qquad \gamma(0)=x, \qquad \gamma(\ell)=y,
\end{equation}
where the estimate constant depends only on
\(d\in\mathcal{D}(\Gamma)\). \emph{Rough geodesics} and \emph{roughly
  geodesic rays} are defined by a similar condition. When speaking
about roughly geodesic segments/rays/lines we will always assume that
they satisfy the uniform estimate~\eqref{eq:rough-geo-uniform-est} with the fixed constant depending only on \(d\).

\subsection{The Gromov boundary}
\label{sec:gromov-boundary}

Let \((X,d)\) be a roughly geodesic hyperbolic metric space. Two roughly geodesic rays \(\gamma_{1},\gamma_{2}\colon [0,\infty)\to X\) are \emph{asymptotic} if the distance \(d(\gamma_{1}(t),\gamma_{2}(t))\) is bounded.
The quotient of the set of all roughly geodesic rays in \(X\) by the asymptoticity relation relation is called the \emph{Gromov boundary} of \(X\) and denoted by \(\bd X\). It carries a family of \emph{visual metrics}, depending on \(d\) and a real parameter \(\epsilon > 0\). We say that \(d_{\epsilon}\) is a visual metric on \(\bd X\) if
\begin{equation}\label{eq:visual-metric-def}
  d_{\epsilon}(\xi,\eta)\asymp e^{-\epsilon (\xi,\eta)},
\end{equation}
where the extension of the Gromov product to \(X\cup \bd X\) is defined as
\begin{equation}
  (\xi,\eta) = \inf\{ \liminf_{t\to\infty} (\gamma_{1}(t),\gamma_{2}(t)) : \xi=[\gamma_{1}], \eta=[\gamma_{2}] \},
\end{equation}
assuming that elements of \(X\) are represented by corresponding constant functions. It is important to note that the diameter of set under the infimum operator is bounded by some constant, so if we are interested in the value of the Gromov product only up to an additive constant, instead of infimum we may just take the limit over some arbitrarily chosen representatives \cite[Remark 3.17]{Bridson1999}. 
Such a metric always exists for sufficiently small \(\epsilon\). Moreover, as a topological space, the Gromov boundary does not depend on the choice of \(\epsilon\), and the metric \(d\in\mathcal{D}(\Gamma)\). 

Finally, if for \(x_{n}\in X\), and \(\xi\in\bd X\) we define \(\lim x_{n}=\xi\) if and only if \(\lim (x_{n},\xi)=\infty\), then \(\ov{X}=X \cup \bd X\) becomes a compactification of \(X\).

Now, let \(\Gamma\) be a hyperbolic group, and let \(d\in\mathcal{D}(\Gamma)\). By the above discussion, this allows us to endow \(\Gamma\) with its boundary \(\bd\Gamma\), independent from \(d\) as a topological space, and a visual metric on \(\bd\Gamma\). The action of \(\Gamma\) on itself by left translations is isometric, so it extends to an action on the set of roughly geodesic rays, compatible with the relation of asymptoticity, and descends to the quotient \(\bd \Gamma\). The action of \(\Gamma\) on \(\bd \Gamma\) satisfies
\begin{equation}
  \abs{(g\xi,g\eta) - (\xi,\eta)} \leq C_{g}
\end{equation}
for some constant \(C_{g}\) depending only on \(g\), and therefore \(\Gamma\) acts on \((\bd \Gamma,d_{\epsilon})\) continuously, by bi-Lipschitz maps.

\subsection{The Patterson-Sullivan measure}
\label{sec:patt-sull-meas}

The space \((\bd \Gamma,d_{\epsilon})\) is a compact metric space, and therefore it admits a Hausdorff measure of dimension \(D\). It is nonzero, finite, and finitely-dimensional, and in the context of hyperbolic geometry it is known as the \emph{Patterson-Sullivan} measure. We will denote it by \(\mu_{PS}\), and normalize it so that \(\mu_{PS}(\bd\Gamma) = 1\). Usually all measures of the form \(\rho d\mu_{PS}\) with \(0<c\leq \rho(\xi)\leq C<\infty\) for some \(c,C\) are called Patterson-Sullivan measures. This class of measures is independent of the choice of \(\epsilon\), but different metrics \(d\in \mathcal{D}(\Gamma)\) usually give rise to mutually singular measures. Moreover, the Hausdorff dimension changes with \(\epsilon\), although the relationship is very simple, namely
\begin{equation}
  e^{\epsilon D}=\omega,
\end{equation}
where \(\omega\) is defined by
\begin{equation}\label{eq:omega-def}
  \omega = \lim_{R\to\infty}\lvert \{ g\in \Gamma :\abs{g}\leq R  \} \rvert^{1/R}.
\end{equation}

Since \(\Gamma\) acts on its boundary by bi-Lipschitz maps, the Patterson-Sullivan measure \(\mu_{PS}\) is quasi-invariant. It actually satisfies a stronger condition of quasi-con\-for\-ma\-li\-ty, namely
\begin{equation}
  \frac{dg_{*}\mu_{PS}}{d\mu_{PS}}(\xi)\asymp \omega^{2(g,\xi)-\abs{g}}.
\end{equation}
It is also Ahlfors regular of dimension \(D\), i.e.\ for \(r \leq \diam{\bd \Gamma}\) we have the following estimate for the volumes of balls:
\begin{equation}
  \mu_{PS}(B_{\bd \Gamma}(r,\xi)) \asymp r^{D}.
\end{equation}
Finally, the Patterson-Sullivan meaure is ergodic for the action of \(\Gamma\).

The theory of Patterson-Sullivan measures for geodesic hyperbolic spaces was developed in \cite{Coornaert1993}, and it was extended to the roughly geodesic case in \cite{Blachere2011}.

\subsection{Boundary representations of hyperbolic groups}
\label{sec:bound-repr-hyperb}

Given a measure space \((X,\mu)\) with a measure class preserving action of a group \(G\), we may define a unitary representation \(\pi_{\mu}\) of \(G\) on the Hilbert space \(L^{2}(X,\mu)\) with the formula
\begin{equation}
  [\pi_{\mu}(g)v](x) = \left(\frac{dg_{*}\mu}{d\mu}(x)\right)^{1/2}v(g^{-1}x)
\end{equation}
for \(g\in G\), \(v\in L^{2}(X,\mu)\), and \(x\in X\).
It is often referred to as the \emph{quasi-regular representation} or \emph{Koopman representation}. Its unitary equivalence class depends only on the equivalence class of \(\mu\). If we apply this construction to the action of \(\Gamma\) on \((\bd \Gamma,\mu_{PS})\), where \(\mu_{PS}\) is the Patterson-Sullivan measure corresponding to a metric \(d\in\mathcal{D}(\Gamma)\), we obtain the \emph{boundary representation} of \(\Gamma\) corresponding to the metric \(d\).

In \cite{Garncarek2014} irreducibility and equivalence of boundary representations of hyperbolic groups were studied, leading to the following results.

\begin{theorem}[{\cite[Theorem 6.2 and 7.4]{Garncarek2014}}]
  For any hyperbolic group \(\Gamma\) which is not virtually cyclic,
  and any metric \(d\in\mathcal{D}(\Gamma)\), the corresponding
  boundary representation is irreducible. 
  Moreover, if \(d,d'\in \mathcal{D}(\Gamma)\) are two metrics giving
  rise to Patterson-Sullivan measures \(\mu_{PS}\) and \(\mu'_{PS}\), and
  boundary representations \(\pi\) and \(\pi'\), then the following
  conditions are equivalent.
  \begin{enumerate}
  \item the metrics \(d\) and \(d'\) are \emph{roughly similar}, i.e. there exists a constant \(\alpha > 0\) such that the difference \(\abs{d-\alpha d'}\) is uniformly bounded,
  \item the Patterson-Sullivan  measures \(\mu_{PS}\) and \(\mu'_{PS}\) are equivalent,
  \item the boundary representations \(\pi\) and \(\pi'\) are unitarily equivalent.
  \end{enumerate}
\end{theorem}

\subsection{Shadows}
\label{sec:preliminaries}

Let \(\Gamma\) be a hyperbolic group endowed with a metric \(d\in \mathcal{D}(\Gamma)\). We will now introduce the setting used in \cite{Garncarek2014}. First, for any \(g\in \Gamma\) we may choose \(\hat{g}\in \bd\Gamma\) such that \((g,\hat{g})\approx \abs{g}\), where the symbol \(\approx\) means that the difference of both sides is bounded by a constant (depending only on the metric group \((\Gamma,d)\)). This defines an `approximate' projection of \(\Gamma\) onto its boundary, which despite the arbitrary choices in its definition, turns out to be very nice in the topological sense, as the next lemma explains. 

\begin{lemma} \label{prop:retraction-onto-boundary}
  The map \(p\colon \ov\Gamma\to\bd\Gamma\) given by
  \begin{equation}
    p(x)=
    \begin{cases}
      \hat{x} & \text{if \(x\in\Gamma\),} \\
      x & \text{if \(x\in\bd\Gamma\)}
    \end{cases}
  \end{equation}
is a continuous retraction. 
\end{lemma}

\begin{proof}
  Since \(\Gamma\) is discrete, \(p\) is clearly continuous at every point of \(\Gamma\), so we just need to check continuity at every \(\xi\in\bd\Gamma\). Let \(x_{n}\subseteq\ov\Gamma\) be a sequence converging to \(\xi\). We have, for \(x_{n}\in\Gamma\),
  \begin{equation}
    (p(x_{n}),\xi) \agtr \min\{ (p(x_{n}),x_{n}), (x_{n},\xi)\} \agtr \min\{ \abs{x_{n}},(x_{n},\xi)\} = (x_{n},\xi).
  \end{equation}
  But \(x_{n}\to\xi\) is equivalent to \((x_{n},\xi)\to\infty\), which implies \((p(x_{n}),\xi) \to \infty\), i.e.\ \(p(x_{n})\to \xi\). For \(x_{n}\in \bd\Gamma\) we have \(p(x_{n})=x_{n}\), so the convergence is clear.
\end{proof}

Let us now denote \(\check{g}=\widehat{g^{-1}}\). In \cite{Garncarek2014} the \emph{double shadows} \(\Sigma^{2}(g,\rho)\subseteq \bd\Gamma\times\bd\Gamma\) of elements of \(\Gamma\) were introduced. They are neighborhoods of \((\hat{g},\check{g})\) defined as
\begin{equation}
  \Sigma^{2}(g,\rho) = B_{\bd\Gamma}\Big(\hat{g},e^{-\epsilon(\abs{g}/2-\rho)}\Big)\times B_{\bd\Gamma}\Big(\check{g},e^{-\epsilon(\abs{g}/2-\rho)}\Big).
\end{equation}
If for \(R,h\geq 0\) we introduce the annulus \(A_{R,h}\) of radius \(R\) and thickness \(2h\) (where \(d\) refers to the metric used, and will be omitted from notation whenever a metric is fixed by the context) defined by
\begin{equation}
  A_{R,h} = \{ g\in\Gamma : \abs{g} \in [R-h,R+h]\},
\end{equation}
we can formulate following fundamental result.

\begin{proposition}[{\cite[Proposition 4.5]{Garncarek2014}}] \label{prop:double-shadows-cover}
  For sufficiently large \(\rho\) and \(h\) the family of double
  shadows
  \begin{equation}
    \{\Sigma^{2}(g,\rho) : g\in A_{R,h}\}
  \end{equation}
  is a cover of \(\bd\Gamma\times\bd\Gamma\) for all \(R > 0\).
\end{proposition}

Henceforth, we will fix \(\rho\) and \(h\)
sufficiently large for this result to hold, and such that  the estimate
\begin{equation}
  \card{A_{R,h}}  \asymp \omega^{R}
\end{equation}
is satisfied. We will omit them from notation, by writing \(A_{R}\), and
\(\Sigma^{2}(g)\).

\subsection{Property RD}
\label{sec:property-rd}

Let \(G\) be a finitely generated group. Denote by \(\abs{x}\) the word-length of \(x\in G\) corresponding to a fixed generating set. For \(f\in\CC[G]\) denote
\begin{equation}
  r(f) = \max\{ \abs{x} : f(x)\ne 0\}.
\end{equation}
This quantity measures how far the support of \(f\) is from the unit element of \(G\). 

The multiplication in the group algebra \(\CC[G]\) extends to an action of \(\CC[G]\) on its completion \(\ell^{2}(G)\) by convolution. The standard definition of the property of rapid decay, abbreviated to `property RD' reads as follows. We say that \(G\) has \emph{property RD} if there exists a polynomial \(P\) with real coefficients, such that for any \(f\in\CC[G]\) and \(g\in \ell^{2}(G)\) we have
\begin{equation}
  \norm{f*g}_{2}\leq P(r(f)) \norm{f}_{2}\norm{g}_{2}.
\end{equation}
Clearly, this does not depend on the choice of the generating set of \(G\), and the polynomial \(P\) can be chosen to be of the form \(P(t)=C(1+t)^{n}\) for some \(C>0\).

Property RD can be reformulated in terms of matrix coefficients of unitary representations of \(G\). First, we say that a unitary representation \(\pi\colon G\to \mathcal{U}(\mathcal{H})\) is \emph{weakly contained in the regular representation}, if for all \(f\in \CC[G]\) we have
\begin{equation}
  \norm{\pi(f)}_{op} \leq \norm{f}_{op},
\end{equation}
where \(\pi\) is uniquely extended to \(\CC[G]\) by linearity, and on
the right we treat \(f\) as its corresponding convolution operator,
i.e. \(\lambda(f)\), where \(\lambda\) is the left regular
representation. Now, a finitely generated group \(G\) has property RD
if and only if there exist constants \(C,s>0\) such that for every
 representation \(\pi\colon G \to \mathcal{U}(\mathcal{H})\) weakly contained in the regular representation,
and for every \(\xi,\eta\in\mathcal{H}\) the estimate
  \begin{equation}\label{eq:decay-condition}
    \left( \sum_{x\in G} \frac{\abs{\langle \pi(x)\xi,\eta\rangle}^{2}}{(1+\abs{x})^{s}} \right)^{1/2} \leq C\norm{\xi}\norm{\eta}
  \end{equation}
is satisfied.

The crucial fact we will be relying on is the following strengthening of the standard property RD for hyperbolic groups, where thanks to restricting to the annulus, one can decrease the exponent (without this restriction, hyperbolic groups satisfy \eqref{eq:decay-condition} with exponent \(s=3\)).

\begin{proposition}
  \label{prop:annular-RD-matrix-coeffs}
  Let \(\Gamma\) be a hyperbolic group endowed with a metric \(d\in\mathcal{D}(\Gamma)\), and let \(\pi\) be a unitary representation of \(\Gamma\) on a Hilbert space \(\mathcal{H}\), weakly contained in the regular representation. Then for \(v,w \in \mathcal{H}\) the estimate
  \begin{equation}
    \left(\sum_{g\in A_{R,h}} \abs{\langle \pi(g)v,w \rangle}^{2} \right)^{1/2} \lasymp_{h} (1+R)\norm{v}\norm{w}
  \end{equation}
  holds uniformly in \(R\).
\end{proposition}

In a less general setting of a group acting on a negatively curved manifold, this result appears in \cite[Section 3.2]{Jolissaint1990}, and the proof presented therein requires only some technical adjustments to work for a general hyperbolic group \(\Gamma\) equipped with a metric \(d\in\mathcal{D}(\Gamma)\). We will include the adapted proof for the sake of self-containment. Note that a proof of the fact that hyperbolic groups satisfy property RD can be fund in \cite{delaHarpe1988}.

Before we proceed to the proof, note that since by \cite{Adams1994} the action of \(\Gamma\) on \(\bd \Gamma\) is amenable, and by \cite{Kuhn1994}, ergodic amenable actions lead to quasi-regular representations which are weakly contained in the regular representation, Proposition \ref{prop:annular-RD-matrix-coeffs} applies to the matrix coefficients of boundary representations.

We begin the proof of Proposition \ref{prop:annular-RD-matrix-coeffs} with several technical lemmas.

\begin{lemma} \label{prop:annular-rd-lem1}
  Let \(g\in \Gamma\). Then
  \begin{equation}
    \diam \{ x\in A_{R,h} : x^{-1}g \in A_{R',h'}\} \aless_{h,h'} \max\{ 0, R+R'-\abs{g} \}.
  \end{equation}
  In particular, for fixed \(h\) and \(h'\) the cardinality of this set is bounded uniformly in \(R+R'-\abs{g}\).
\end{lemma}

\begin{proof}
  Assume that the considered set is non-empty, and contains two elements \(x_{1},x_{2}\). We have
  \begin{equation}
    (x_{1},x_{2}) \agtr \min\{ (x_{1},g), (x_{2},g)\} \approx_{h,h'} \frac{1}{2}(R + \abs{g} - R'),
  \end{equation}
  and therefore
  \begin{equation}
    d(x_{1},x_{2}) = \abs{x_{1}} + \abs{x_{2}} - 2(x_{1},x_{2}) \approx_{h,h'} R+R'-\abs{g}. \qedhere
  \end{equation} 
\end{proof}

\begin{lemma}\label{prop:annular-rd-lem2}
  Suppose that \(\phi,\psi\in \CC[\Gamma]\) are supported in \(A_{R,h}\) and \(A_{R',h'}\), respectively. Then for \(g\in \Gamma\) we have
  \begin{equation}
    \abs{\phi*\psi(g)}^{2}\lasymp_{h,h',R+R'-\abs{g}} \abs{\phi}^{2}*\abs{\psi}^{2}(g).
  \end{equation}
\end{lemma}

\begin{proof}\ 
  We have
  \begin{equation}
    \phi*\psi(g) = \sum_{x} \phi(x)\psi(x^{-1}g),
  \end{equation}
  where \(x\) runs through the set \(\{ x\in A_{R,h} : x^{-1}g\in A_{R',h'}\}\), whose cardinality is, by Lemma~\ref{prop:annular-rd-lem1}, bounded by some \(C\) depending only on \(h, h'\), and \(R+R'-\abs{g}\). Hence,
  \begin{equation}
    \abs{\phi*\psi(g)}^{2} \leq C \sum_{x} \abs{\phi(x)}^{2}\lvert{\psi(x^{-1}g)}\rvert^{2} = C \abs{\phi}^{2}*\abs{\psi}^{2}(g).\qedhere
  \end{equation}
\end{proof}

\begin{lemma} \label{prop:element-decomposition}
  Let \(g\in A_{R,h}A_{R',h'}=\{ kk' : k\in A_{R,h}, k'\in A_{R',h'}\}\) for some \(R,R',h,h'\geq 0\), and put \(p=(R+R'-\abs{g})/2\). Then there exist \(u_{g},v_{g}\in \Gamma\) such that
  \begin{enumerate}
  \item \(u_{g}v_{g}=g\),\
  \item \(\lvert{u_{g}}\rvert\approx_{h,h'} R-p\), and \(\lvert{v_{g}}\rvert\approx_{h,h'} R'-p\),
  \item for any decomposition \(g=uv\) with \(u\in A_{R,h}\), and \(v\in A_{R',h'}\), we have \(\lvert{u_{g}^{-1}u}\rvert=\lvert{vv_{g}^{-1}}\rvert\approx_{h,h'}p\).
  \end{enumerate}
\end{lemma}

\begin{proof}
  By the triangle inequality we have \(\abs{g}\agtr_{h,h'} R'-R\), so \(R-p \agtr_{h,h'} 0\). We may thus pick \(u_{g}\) satisfying \(\abs{u_{g}}\approx_{h,h'} R-p\) on some roughly geodesic segment from \(1\) to \(g\), and define \(v_{g} = u_{g}^{-1}g\). Then
  \begin{equation}
    \abs{v_{g}} \approx \abs{g}-\abs{u_{g}} \approx_{h,h'} R'-p.
  \end{equation}
  To verify condition (3), choose a decomposition \(g=uv\) with  \(u\in A_{R,h}\), and \(v\in A_{R',h'}\), and observe that \(u_{g}^{-1}u = (vv_{g}^{-1})^{-1}\), so we have to estimate only \(\lvert{u_{g}^{-1}u}\rvert\). We have
  \begin{equation}
    \abs{u_{g}}\geq (u_{g},u) \agtr \min\{ (u_{g},g), (g,u)\},
  \end{equation}
  where
  \begin{equation}
    (u_{g},g)\approx \abs{u_{g}}\approx_{h,h'} R-p,
  \end{equation}
  and
  \begin{equation}
    (g,u) = \frac{1}{2}(\abs{g}+\abs{u}-d(g,u)) \approx_{h,h'} R-p,
  \end{equation}
  so 
  \begin{equation}
    (u_{g},u) \approx_{h,h'} R-p.
  \end{equation}
  It follows that
  \begin{equation}
    \lvert u_{g}^{-1}u \rvert = \lvert u_{g}\rvert + \abs{u} - 2(u_{g},u) \approx_{h,h'} p. \qedhere
  \end{equation}
\end{proof}

\begin{lemma} \label{prop:convolution-estimate-annuli}
  Suppose that \(\phi,\psi\in \CC[\Gamma]\) are supported in \(A_{R,h}\) and \(A_{R',h'}\), respectively. Then, the  restriction of their convolution to the annulus \(A_{R'',h''}\) satisfies the estimate
  \begin{equation}
    \norm{ \phi*\psi \res_{A_{R'',h''}}}_{2} \lasymp_{h,h',h''} \norm{\phi}_{2}\norm{\psi}_{2}.
  \end{equation}
\end{lemma}

\begin{proof}
  Pick \(g\in A_{R'',h''}\) such that \(\phi*\psi(g)\ne 0\); in particular, \(g\in A_{R,h}A_{R',h'}\). Let \(p=(R+R'-R'')/2\), and let \(g=u_{g}v_{g}\) be the decomposition from Lemma~\ref{prop:element-decomposition}. Then there exists a constant \(\kappa\geq 0\) depending only on \(h,h'\) and \(h''\) such that
  \begin{equation}
    \abs{\phi*\psi(g)} \leq \sum_{uv=g}\abs{\phi(u)\psi(v)} \leq \sum_{k\in A_{p,\kappa}} \abs{\phi(u_{g}k)\psi(k^{-1}v_{g})} \leq \phi_{p}(u_{g})\psi_{p}(v_{g}),
  \end{equation}
  where
  \begin{equation}
    \phi_{p}(u) = \left(\sum_{k\in A_{p,\kappa}} \abs{\phi(uk)}^{2}\right)^{1/2}\quad\text{and}\quad \psi_{p}(v) = \left(\sum_{k\in A_{p,\kappa}} \abs{\psi(k^{-1}v)}^{2}\right)^{1/2}
  \end{equation}
  are defined on the annuli of the form \(A_{R-p,r}\) and \(A_{R'-p,r'}\) respectively, where \(r\) and \(r'\) depend only on \(h,h',h''\).
  Therefore, we have
  \begin{equation}
    \begin{split}
      \norm{ \phi*\psi \res_{A_{R'',h''}}}_{2}^{2} & = \sum_{g\in
        A_{R'',h''}} \abs{\phi*\psi(g)}^{2} \leq \sum_{g\in
        A_{R'',h''}}\phi_{p}(u_{g})^{2}\psi_{p}(v_{g})^{2} \leq \\
      & \leq \sum_{g\in
        A_{R'',h''}} (\phi_{p}*\psi_{p}(g))^{2}.
    \end{split}
  \end{equation}
  Since \(R-p+R'-p \approx_{h,h',h''} R''\), by Lemma~\ref{prop:annular-rd-lem2} we have
  \begin{equation}
    \sum_{g\in A_{R'',h''}} (\phi_{p}*\psi_{p}(g))^{2} \lasymp_{h,h'} \sum_{g\in A_{R'',h''}} \phi_{p}^{2}*\psi_{p}^{2}(g) = \norm{\phi_{p}}_{2}^{2} \norm{\psi_{p}}_{2}^{2}.
  \end{equation}
  But by Lemma \ref{prop:annular-rd-lem1}
  \begin{equation}
    \norm{\phi_{p}}_{2}^{2} = \sum_{u\in A_{R-p,r}} \sum_{h\in A_{p,\kappa}} \abs{\phi(uh)}^{2} \lasymp_{r,\kappa} \norm{\phi}_{2}^{2},
  \end{equation}
  and similarly with \(\psi_{p}\), thus finishing the proof.
\end{proof}

Using the lemmas above, we may prove the following estimate.

\begin{proposition}
  \label{prop:annular-RD-hyp}
  Let \(\Gamma\) be a hyperbolic group endowed with a metric \(d\in \mathcal{D}(\Gamma)\), and let \(R,h>0\). Then for \(\phi\in \CC[\Gamma]\) supported in \(A_{R,h}\) and \(\psi\in \ell^{2}(\Gamma)\) we have
  \begin{equation}
    \norm{\phi * \psi}_{2} \lasymp_{h} (1+R)\norm{\phi}_{2}\norm{\psi}_{2}
  \end{equation}
\end{proposition}

\begin{proof}
  Without loss of generality we may assume that \(\phi\) and \(\psi\) take positive real values. If we denote by \(\chi_{n}\) the characteristic function of the annulus \(\{ g\in \Gamma: \abs{g}\in [n,n+1)\}\), and \(\psi_{n} = \psi\chi_{n}\), then for \(n\in I_{m} = [m-R-h-1,m+R+h] \) we have by Lemma \ref{prop:convolution-estimate-annuli}
  \begin{equation}
    \norm{(\phi*\psi_{n})\chi_{m}}_{2} \lasymp_{h}\norm{\phi}_{2}\norm{\psi_{n}}_{2},
  \end{equation}
  and for \(n\) outside this interval the corresponding norm is \(0\).
  Hence
  \begin{equation}
    \begin{split}
      \norm{(\phi*\psi)\chi_{m}}_{2} & \leq \sum_{n\in
        I_{m}}\norm{(\phi*\psi_{n})\chi_{m}} \lasymp_{h} \sum_{n\in
        I_{m}} \norm{\phi}_{2} \norm{\psi_{n}}_{2} \lasymp_{h} \\
      & \lasymp_{h}
      (1+R)^{1/2}\norm{\phi}_{2}\left( \sum_{n\in
          I_{m}}\norm{\psi_{n}}_{2}^{2} \right)^{1/2},
    \end{split}
  \end{equation}
  and therefore
  \begin{equation}
    \begin{split}
      \norm{\phi*\psi}_{2}^{2} & = \sum_{m}
      \norm{(\phi*\psi)\chi_{m}}_{2}^{2} \lasymp_{h}
      (1+R)\norm{\phi}_{2}\sum_{m}\sum_{n\in I_{m}}
      \norm{\psi_{n}}_{2}^{2} \lasymp_{h} \\
      & \lasymp_{h} (1+R)^{2}\norm{\phi}_{2} \sum_{n}\norm{\psi_{n}}_{2}^{2} = (1+R)^{2}\norm{\phi}_{2}^{2}\norm{\psi}_{2}^{2}. \qedhere
    \end{split}
  \end{equation}
\end{proof}

The proof of Proposition~\ref{prop:annular-RD-matrix-coeffs} now becomes a simple calculation.

\begin{proof}[Proof of Proposition~\ref{prop:annular-RD-matrix-coeffs}]
  Denote \(\phi(g)=\conj{\langle \pi(g)v,w\rangle}\) for \(g\in A_{R,h}\), and \(\phi(g)=0\) otherwise. The expression we are to estimate is then the \(\ell^{2}\)-norm of \(\phi\). We get, using Proposition~\ref{prop:annular-RD-hyp}
  \begin{equation}
    \begin{split}
      \norm{\phi}_{2}^{2} & = \sum_{g\in \Gamma} \phi(g)
      \langle\pi(g)v,w\rangle = \langle \pi(\phi) u,v\rangle \leq \\ & \leq
      \norm{\pi(\phi)}_{op}\norm{v}\norm{w} 
      \lasymp_{h}(1+R)\norm{\phi}_{2}\norm{v}\norm{w},
    \end{split}
  \end{equation}
  and consequently \(\norm{\phi}_{2}\lasymp_{h}(1+R)\norm{v}\norm{w}\).
\end{proof}

\section{Equidistribution and asymptotic Schur orthogonality}
\label{sec:equid-asympt-schurs}

Let \(\Gamma\) be a hyperbolic group endowed with a metric \(d\in\mathcal{D}(\Gamma)\), and the corresponding visual metric  \(d_{\epsilon}\) on the boundary \(\bd\Gamma\), where \(\epsilon>0\) is fixed. Denote by \(\mu_{PS}\) the associated Patterson-Sullivan measure, and let \(\mathcal{H}=L^{2}(\bd\Gamma,\mu_{PS})\).

In this section we first deal with the distribution of the points of the form \((\hat{g},\check{g})\) in \(\bd\Gamma\times\bd\Gamma\). We show that in a certain sense, they are approximately uniformly distributed. Namely, to the points \((\hat{g},\check{g})\) with \(g\in A_{R}\), we may assign weights in a controlled manner, in such a way that the weighted average of the corresponding Dirac masses gives a good approximation of \(\mu_{PS} \otimes\mu_{PS}\).

Using this equidistribution result, we prove that the matrix coefficients of a boundary representation of \(\Gamma\) satisfy certain \emph{asymptotic Schur orthogonality relations} resembling Schur  orthogonality relations.

\subsection{Equidistribution}
\label{sec:equidistribution}

Define the maps \(J\colon \Gamma\to\ov\Gamma\times\ov\Gamma\) and  \(J'\colon \Gamma\to\bd\Gamma\times\bd\Gamma\) by
\begin{equation}
  J(g)=(g,g^{-1})\qquad\text{and}\qquad J'(g)=(\hat{g},\check{g}).
\end{equation}
By Lemma~\ref{prop:retraction-onto-boundary} these maps are both continuous.

We begin by showing that two ways of approximating a measure of the form \(\mu\otimes\mu\) on \(\bd\Gamma\times\bd\Gamma\)  are equivalent.

\begin{lemma} \label{prop:asymp-uniform-equiv-conditions}
  Let \(\{\mu_{R} : R\geq 0\}\) be a one-parameter family of probability measures on \(\Gamma\), and let \(\mu\in\Prob(\bd\Gamma)\) be atom-free. Then the following conditions are equivalent
  \begin{enumerate}
  \item in the space \(C(\ov\Gamma\times\ov\Gamma)^{*}\) endowed with
    the weak-* topology we have
    \begin{equation}
      \lim_{R\to\infty} J_{*}\mu_{R} = \mu \otimes \mu.
    \end{equation}
  \item in the space \(C(\bd\Gamma\times\bd\Gamma)^{*}\) endowed with
    the weak-* topology we have
    \begin{equation}
      \lim_{R\to\infty} J'_{*}\mu_{R} = \mu \otimes \mu.
    \end{equation}
  \end{enumerate}
\end{lemma}

\begin{proof} 
  First, assume (1), and let \(F\in C(\bd\Gamma\times \bd\Gamma)\). Define \(\tilde{F}\in C(\ov\Gamma\times\ov\Gamma)\) using the retraction \(p\colon \ov\Gamma\to\bd\Gamma\) from Lemma~\ref{prop:retraction-onto-boundary}, as \(\tilde{F}(x,y)=F(p(x),p(y))\). By (1), and the fact that \(p\) is identity on \(\bd\Gamma\), we have
  \begin{equation}
    \begin{split}
      \lim_{R\to\infty} \int F\,dJ'_{*}\mu_{R} & =
      \lim_{R\to\infty} \int \tilde{F}\,dJ_{*}\mu_{R} =  \int
      F(\xi,\eta)\,d\mu(\xi)d\mu(\eta).
    \end{split}
  \end{equation}

  Let us now prove the other implication. Assume (2), and take \(F\in C(\ov\Gamma\times\ov\Gamma)\). Define \(\tilde{F}\) in the same manner as before. We then have
  \begin{equation}
    \abs{\int F\,dJ_{*}\mu_{R} - \int F\,dJ'_{*}\mu_{R}} \leq \int \lvert{ F - \tilde{F}}\rvert \,dJ_{*}\mu_{R},
  \end{equation}
  but \(G=\lvert F-\tilde{F}\rvert\) vanishes on \(\bd\Gamma\times\bd\Gamma\), and we may estimate the above integral separately on some large ball \(B=B_{\Gamma}(1,r)\subseteq \Gamma\), and its complement, by writing
  \begin{equation} \label{eq:proof-auec-1}
      \int G\,dJ_{*}\mu_{R}  = \int_{B} G\circ J\,d\mu_{R} + \int_{\Gamma\setminus B} G\circ J\,d\mu_{R}. 
  \end{equation}
  The first term can be bounded by
  \begin{equation}
    \int_{B} G\circ J\,d\mu_{R} \leq \norm{G}_{\infty} \mu_{R}(B).
  \end{equation}
  To see that \(\lim_{R\to\infty}\mu_{R}(B)=0\), observe that since \(\mu\) is atom-free, and \(\ov\Gamma\) is a compact Hausdorff space, for any \(\lambda>0\) there exists a continuous function \(H\colon \ov\Gamma\times\ov\Gamma \to [0,1]\) such that \(H(\hat{g},\check{g}) = 1\) for \(g\in B\), and
  \begin{equation}
    \int H\,d\mu^{2} < \lambda.
  \end{equation}
  Then
  \begin{equation}
    \mu_{R}(B) = \int_{B}H(\hat{g},\check{g})\,d\mu_{R}(g) \leq\int H\,dJ'_{*}\mu_{R},
  \end{equation}
  and therefore
  \begin{equation}
    \lim_{R\to\infty} \mu_{R}(B) \leq \lim_{R\to\infty} \int H\, dJ'_{*}\mu_{R}(g) = \int H\,d\mu^{2} < \lambda
  \end{equation}
  for any ball \(B\subseteq \Gamma\).

  To bound the second term in \eqref{eq:proof-auec-1}, observe that for any \(\lambda>0\) the set
  \begin{equation}
    U_{\lambda} = \{ (x,y) \in \ov\Gamma \times \ov\Gamma : G(x,y) < \lambda \}
  \end{equation}
  is an open neighborhood of \(\bd\Gamma\times\bd\Gamma\), and since
  \(\bd\Gamma\times\bd\Gamma\) can be written as an intersection of closed sets
  \begin{equation}
    \bd\Gamma\times\bd\Gamma = \bigcap_{r>0}
    \left((\ov{\Gamma} \setminus B_{\Gamma}(1,r))\times
      (\ov{\Gamma}\setminus B_{\Gamma}(1,r))\right),
  \end{equation}
  by compactness every \(U_{\lambda}\) contains one of these sets. This means that for sufficiently large \(r\)
  \begin{equation}
    \int_{\Gamma\setminus B} G\circ J \,d\mu_{R} < \lambda \mu_{R}(\Gamma\setminus B) \leq \lambda.
  \end{equation}
  This gives estimates in \eqref{eq:proof-auec-1}, and proves that (2) implies (1).
\end{proof}

Now, we may proceed to the equidistribution theorem, asserting that for a hyperbolic group, the product measure \(\mu_{PS}\otimes\mu_{PS}\) can be approximated in the way described in Lemma~\ref{prop:asymp-uniform-equiv-conditions} by a reasonably well-behaved family of measures on \(\Gamma\). 

\begin{theorem}\label{prop:equidistribution}
  For any hyperbolic group \(\Gamma\) endowed with a metric \(d\in\mathcal{D}(\Gamma)\), there exists a family of measures \(\{\mu_{R} : R\geq 0\}\subseteq \Prob(\Gamma)\) such that
  \begin{enumerate}
  \item \(\lim_{R\to\infty} J_{*}\mu_{R} = \mu_{PS} \otimes \mu_{PS}\) in the
    weak-* topology on \(C(\ov\Gamma\times\ov\Gamma)^{*}\)
  \item there exists \(h>0\) such that every
    \(\mu_{R}\) is supported in the annulus \(A_{R,h}\),
  \item the estimate \(\mu_{R}(\{g\})\lasymp \omega^{-R}\) is
    satisfied uniformly in \(R\).
  \end{enumerate}
\end{theorem}

Notice, that conditions (2) and (3) above are as close as one can get to uniform distribution on the annulus. Indeed, in the context of general hyperbolic groups, from the very beginning all equalities and inequalities are approximate.

\begin{proof}[Proof of Theorem~\ref{prop:equidistribution}]
  The construction of measures \(\mu_{R}\) will proceed analogously to
  the proof of \cite[Proposition 5.4]{Garncarek2014}. Using Proposition \ref{prop:double-shadows-cover}, we may take \(\rho\) and \(h\) sufficiently large for the double shadows \(\Sigma_{2}(g,\rho)\) with \(g\in A_{R,h}\) to cover \(\bd\Gamma\times\bd\Gamma\).
  Fix \(R\) and define a family \(\{E_{g} : g\in A_{R,h}\}\) of subsets
  of \(\bd\Gamma\times\bd\Gamma\) by putting an arbitrary linear order
  on \(A_{R,h}\), and taking inductively 
  \begin{equation} E_{g} = \Sigma^{2}(g) \setminus
    \bigcup_{h<g} E_{h} \subseteq \Sigma^{2}(g).
  \end{equation} These sets form a measurable partition of
  \(\bd\Gamma\times\bd\Gamma\), and we may put
  \begin{equation}
    \mu_{R}(\{g\}) = \mu_{PS}^{2}(E_{g}) \leq \mu^{2}(\Sigma^{2}(g)).
  \end{equation}
  Since \(\mu_{PS}\) is normalized to be a probability measure, conditions (2) and (3) are satisfied.

  It remains to verify condition (1), or equivalently, condition (2) of Lemma~\ref{prop:asymp-uniform-equiv-conditions}. Take \(F\in C(\bd\Gamma\times\bd\Gamma)\). The product \(\bd\Gamma\times\bd\Gamma\) is a compact metric space when endowed with the \(\ell^{\infty}\)-product \(d_{\epsilon}^{2}\) of the visual metrics \(d_{\epsilon}\) on the factors, and therefore \(F\) is uniformly continuous, and admits a modulus of continuity, i.e.\ a non-decreasing continuous function \(m\colon [0,\infty)\to[0,\infty)\) such that \(m(0)=0\), and 
  \begin{equation}
    \abs{ F(\xi,\eta)-F(\xi',\eta')} \leq \max \{ m( d_{\epsilon}(\xi,\xi')) , m(d_{\epsilon}(\eta,\eta'))\}.
  \end{equation}
  We now have
  \begin{equation}
    \int_{\Gamma} F\,dJ'_{*}\mu_{R} = \int_{\Gamma}F\circ J'\,d\mu_{R}= \sum_{g\in A_{R}} F(\hat{g},\check{g})\mu_{R}(\{g\}) =  \sum_{g\in A_{R}}\int_{E_{g}} F(\hat{g},\check{g})\,d\mu_{PS}^{2},
  \end{equation}
  and since for \(g\in A_{R}\) and \(\zeta,\zeta'\in E_{g}\subseteq \Sigma^{2}(g,\rho)\) we have \(d_{\epsilon}^{2}(\zeta,\zeta')\leq Ce^{-\epsilon R/2}\) for some constant \(C\), we finally get
  \begin{equation}
    \begin{split}
      \Big\lvert\int_{\bd\Gamma\times\bd\Gamma} F\,d\mu_{PS}^{2} & -
        \int_{\Gamma} F\,dJ'_{*}\mu_{R}\Big\rvert \leq \\
        & \leq
      \sum_{g\in A_{R}} \int_{E_{g}}\abs{F(\xi,\eta) -
        F(\hat{g},\check{g}) } \,d\mu_{PS}^{2}(\xi,\eta) \leq \\
      & \leq \sum_{g\in A_{R}}  \mu_{PS}^{2}(E_{g}) m(Ce^{-\epsilon R/2}) \lasymp \card{A_{R}}^{-1} \sum_{g\in A_{R}} m(Ce^{-\epsilon R/2}) =\\
      & = m(Ce^{-\epsilon R/2}) \xrightarrow[R\to\infty]{} 0. \qedhere
    \end{split}
  \end{equation}
\end{proof}

\subsection{Asymptotic Schur orthogonality relations}
\label{sec:asympt-orth}

Let us stop for a moment and recall the classical Schur orthogonality relations. Let \(G\) be a finite group, and let \(\pi\) be its irreducible representation on a finite-dimensional Hilbert space \(V\). The matrix coefficients of \(\pi\) are functions \(D_{uv} \in \CC[G]\), indexed by \(u,v\in V\), defined by
\begin{equation}
  D_{uv}(g) = \langle \pi(g)u,v\rangle_{V}.
\end{equation}
As \(G\) is finite, \(\CC[G]\) is naturally a Hilbert space with the normalized inner product
\begin{equation}
  \langle \phi,\psi\rangle = \frac{1}{\card{G}} \sum_{g\in G} \phi(g)\conj{\psi(g)}.
\end{equation}
With this notation, Schur orthogonality (for a single representation) can be rephrased as the identity
\begin{equation}
   \dim{V} \cdot\langle D_{uv}, D_{u'v'}\rangle  =  \langle u,u'\rangle\conj{\langle v,v'\rangle}.
\end{equation}

Now, let us return to our original setting. Let \(\pi\colon \Gamma\to\mathcal{U}(\mathcal{H})\) be the boundary representation of \(\Gamma\) associated to the metric \(d\in\mathcal{D}(\Gamma)\). Also, define
\begin{equation}
  \tilde{\pi}(g)=\frac{1}{\langle \pi(g)\one,\one\rangle}\pi(g).
\end{equation}
The matrix coefficient
\begin{equation}
\Xi(g) = \langle \pi(g)\one,\one\rangle  \in\mathbb{R}^{+}
\end{equation}
appearing in the above normalization is often referred to as the \emph{Harish-Chandra function} of the representation \(\pi\). Recall, that by \cite[Lemma 5.1]{Garncarek2014} it can be estimated by
\begin{equation}
  \Xi(g) \asymp \omega^{-\abs{g}/2}(1+\abs{g})
\end{equation}
uniformly in \(g\in\Gamma\). We clearly have \(\tilde{\pi}(g^{-1})=\tilde{\pi}(g)^{*}\). If we denote, following \cite{Garncarek2014},
\begin{equation}
  \tilde{P}_{g}(\xi) = \frac{1}{\Xi(g)} \left(\frac{dg_{*}\mu}{d\mu}(\xi)\right)^{1/2},
\end{equation}
then
\begin{equation}
  [\tilde{\pi}(g)v](\xi) = \tilde{P}_{g}(\xi)v(g^{-1}\xi).
\end{equation}

With this notation, we can formulate our Asymptotic Orthogonality Theorem as follows.

\begin{theorem}[Asymptotic Orthogonality] \label{prop:asymptotic-orthogonality-general}
Let \(\{\mu_{R}: R>0\}\) be a family of probability measures on \(\Gamma\), satisfying conditions (1)-(3) of Theorem~\ref{prop:equidistribution}. Then for all continuous functions
    \(f_{1},f_{2} \in C(\ov{\Gamma})\), and vectors
    \(v_{1}, v_{2}, w_{1}, w_{2}\in L^{2}(\bd\Gamma,\mu)\)
    we have
    \begin{multline}\label{eq:as-ort-gen}
      \lim_{R\to\infty} \int f_{1}(g)
      f_{2}(g^{-1}) \langle \tilde{\pi}(g)v_{1}, w_{1}\rangle
      \conj{\langle \tilde{\pi}(g)v_{2}, w_{2}\rangle} \,d\mu_{R}(g)= \\ =
      \langle f_{2}|_{\bd\Gamma} v_{1},v_{2}\rangle
      \conj{\langle w_{1},f_{1}|_{\bd\Gamma} w_{2}\rangle}.
    \end{multline}
\end{theorem}

To see how it relates to the Schur orthogonality relations, let us specify to the case where \(f_{1}=f_{2}=\one_{\ov \Gamma}\). If we forget for a moment about the limit, then the left-hand side is an analogue of the scalar product of matrix coefficients of \(\pi\) restricted to an annulus, with some weights added. By passing to the limit, we get an asymptotic version of the orthogonality relation between the matrix coefficients.

The proof of Theorem~\ref{prop:asymptotic-orthogonality-general} will be divided into two lemmas. But first, we will
introduce some auxiliary notation. Fix
\(f_{1},f_{2}\in C(\ov{\Gamma})\), and  denote the expression under
the limit in the left-hand side of \eqref{eq:as-ort-gen} as
\begin{equation}
  \Phi_{R}(v_{1},v_{2},w_{1},w_{2}) = \int f_{1}(g) f_{2}(g^{-1}) \langle \tilde{\pi}(g)v_{1}, w_{1}\rangle \conj{\langle \tilde{\pi}(g)v_{2}, w_{2}\rangle}\,d\mu_{R}(g).
\end{equation}
It can be regarded as a quadrilinear form on \(\mathcal{H} \times \mathcal{H}^{*} \times \mathcal{H}^{*} \times \mathcal{H}\).

\begin{lemma}\label{prop:lemma-as-ort-uni-bound}
  For any \(f_{1},f_{2}\in C(\bd\Gamma)\) the corresponding quadrilinear forms \(\Phi_{R}\) are uniformly bounded with respect to the norm
  \begin{equation}
    \norm{\Phi} = \inf\{ C > 0 : \Phi(v_{1},v_{2},w_{1},w_{2}) \leq C\norm{v_{1}}\norm{v_{2}}\norm{w_{1}}\norm{w_{2}}\}.
  \end{equation}
\end{lemma}

\begin{proof}
  Without loss of generality we may assume that \(\norm{f_{i}}_{\infty}\leq 1\). Then for \(v_{1},v_{2},w_{1},w_{2} \in \mathcal{H}\) we get
  \begin{equation}
    \begin{split}
      \abs{\Phi_{R}(v_{1},v_{2},w_{1},w_{2})} & \leq \int
      \abs{\langle\tilde{\pi}(g)v_{1},w_{1}\rangle \langle
        \tilde{\pi}(g)v_{2},w_{2}\rangle}\,d\mu_{R}(g) = \\
      & = \sum_{g\in A_{R}} \frac{\mu_{R}(g)}{\Xi(g)^{2}} \abs{ \langle \pi(g)v_{1},w_{1}\rangle \langle \pi(g)v_{2},w_{2} \rangle} \lasymp \\
      & \lasymp \frac{1}{(1+R)^{2}}\sum_{g\in A_{R}} \abs{ \langle \pi(g)v_{1},w_{1}\rangle \langle \pi(g)v_{2},w_{2} \rangle}.
    \end{split}
  \end{equation}
  But by Proposition~\ref{prop:annular-RD-matrix-coeffs},
  \begin{equation}
    \begin{split}
      \sum_{g\in A_{R}} \lvert \langle \pi(g)v_{1}, & w_{1}\rangle \langle \pi(g)v_{2},w_{2} \rangle \rvert \leq \\
      & \leq \left(\sum_{g\in A_{R}} \abs{ \langle
          \pi(g)v_{1},w_{1}\rangle}^{2}\right)^{1/2} \left( \sum_{g\in
          A_{R}}\abs{\langle \pi(g)v_{2},w_{2}
          \rangle}^{2}\right)^{1/2} \lasymp \\
      & \lasymp (1+R)^{2}\norm{v_{1}}\norm{w_{1}}\norm{v_{2}}\norm{w_{2}},
    \end{split}
  \end{equation}
  which concludes the proof.
\end{proof}

\begin{lemma} \label{prop:lemma-as-ort-lip-convergence}
  For \(v_{i},w_{i}\in\Lip(\bd\Gamma,d_{\epsilon})\), we have the convergence
  \begin{equation}
    \lim_{R\to\infty} \Phi_{R}(v_{1},v_{2},w_{1},w_{2}) = \langle f_{2}|_{\bd\Gamma} v_{1},v_{2}\rangle
      \conj{\langle  w_{1},f_{1}|_{\bd\Gamma}w_{2}\rangle}.
  \end{equation}
\end{lemma}

\begin{proof}
  Define \(F\in C(\ov{\Gamma}\times\ov{\Gamma})\) by
  \begin{equation}
    F(g,h) = f_{1}(g)f_{2}(h)v_{1}(p(h)) \conj{w_{1}(p(g))v_{2}(p(h))} w_{2}(p(g)),
  \end{equation}
  and put
  \begin{equation}
    \begin{split}
    \alpha_{i}(g) &= \langle \tilde{\pi}(g)v_{i},w_{i} \rangle,\\ \beta_{i}(g) &= v_{i}(\check{g})\conj{w_{i}(\hat{g})}, \\ d\nu_{R}(g)&=\lvert {f_{1}(g)f_{2}(g^{-1})}\rvert d\mu_{R}(g).
    \end{split}
  \end{equation}
  By definition of \(\mu_{R}\) we have
  \begin{equation} \label{eq:as-ort-proof-1}
    \lim_{R\to\infty} \int F\,dJ_{*}\mu_{R} = \int F\, d\mu^{2} = \langle f_{2}|_{\bd\Gamma} v_{1},v_{2}\rangle \conj{\langle  w_{1},f_{1}|_{\bd\Gamma}w_{2}\rangle},
  \end{equation}
  and we need to estimate the difference
  \begin{equation}\label{eq:as-ort-proof-difference}
    \begin{split}
      \Big\lvert \Phi_{R}(v_{1}, & v_{2},w_{1},w_{2})  -
        \int F\,dJ_{*}\mu_{R}\Big\rvert  \leq \\
        &\leq  \int
        \abs{\alpha_{1}(g)\conj{\alpha_{2}(g)} -
          \beta_{1}(g)\conj{\beta_{2}(g)}} \,d\nu_{R}(g).
      \end{split}
  \end{equation}
  By \cite[Lemma 5.3]{Garncarek2014}, for \(v,w \in \Lip(\bd\Gamma,d_{\epsilon})\) we have
  \begin{equation}
    \abs{ \langle \tilde{\pi}(g)v,w\rangle - v(\check{g})\conj{w(\hat{g})} } \lasymp_{v,w} (1+\abs{g})^{-1/D},
  \end{equation}
  so for \(g\in A_{R}\), with estimates depending on \(v_{i}\) and \(w_{i}\),
 \begin{equation}\label{eq:orthogonality-proof-2}
    \begin{split}
      \Big\lvert \alpha_{1}(g)\conj{\alpha_{2}(g)} & -
        \beta_{1}(g)\conj{\beta_{2}(g)} \Big\rvert \leq \\ &
      \leq \abs{ \alpha_{1}(g)\conj{(\alpha_{2}(g)-\beta_{2}(g))}} + \abs{ (\alpha_{1}(g) -
        \beta_{1}(g))\conj{\beta_{2}(g)}} \lasymp \\
      & \lasymp (\abs{\alpha_{1}(g)}+\abs{\beta_{2}(g)})(1+R)^{-1/D} \lasymp \\
      & \lasymp (\abs{\beta_{1}(g)}+\abs{\beta_{2}(g)} + (1+R)^{-1/D})(1+R)^{-1/D}.
    \end{split}
  \end{equation}
    But 
  \begin{equation}
    \abs{\beta_{i}(g)} \leq \norm{v_{i}}_{\infty}\norm{w_{i}}_{\infty},
  \end{equation}
  hence the differences in \eqref{eq:orthogonality-proof-2} converge to \(0\) uniformly as \(R\to \infty\),  and since \(\nu_{R}(\Gamma)\) is bounded by \(\norm{f_{1}}_{\infty}\norm{f_{2}}_{\infty}\),  so is the difference in \eqref{eq:as-ort-proof-difference}.
\end{proof}

\begin{proof}[Proof of the Asymptotic Orthogonality Theorem \ref{prop:asymptotic-orthogonality-general}]
By Lemma \ref{prop:lemma-as-ort-uni-bound}, the forms \(\Phi_{R}\) are uniformly bounded. Hence, to show the desired convergence, it suffices to do so for vectors belonging to some dense subspaces of \(\mathcal{H}\). The space of Lipschitz functions is such a space, and by Lemma \ref{prop:lemma-as-ort-lip-convergence}, the convergence holds. We are thus done with the proof.
\end{proof}

As an immediate corollary, we deduce an ergodic theorem \`a la Bader-Muchnik for general hyperbolic groups. For \(f\in C(\bd\Gamma)\) let \(m(f)\) denote the corresponding multiplication operator on \(\mathcal{H}\), and let \(P\) be the orthogonal projection onto the subspace spanned by \(\one_{\bd\Gamma}\).

\begin{corollary}
 Let \(\Gamma\) be a hyperbolic group endowed with a metric \(d\in\mathcal{D}(\Gamma)\), and let $\mu_{R}\in \Prob(\Gamma)$ be a family of measures satisfying the conditions (1)-(3) of Theorem~\ref{prop:equidistribution}. Then for any function $f\in C(\ov{\Gamma})$ we have 
 \begin{equation}
   \lim_{R\to \infty}\sum_{g \in \Gamma}\mu_{R}(g)f(g) \tilde{\pi}(g)\to m(f\res_{\bd\Gamma})P
 \end{equation}
in the weak operator topology.
\end{corollary}  
 
Note that this ergodic theorem can be seen as a special case of \cite[Proposition 5.4]{Garncarek2014}, and is sufficient for proving the irreducibility of boundary representations associated with Patterson-Sullivan measures of hyperbolic groups.

\section{Convex cocompact groups of isometries of CAT(-1) spaces}
\label{sec:conv-cocomp-groups}

CAT(-1) spaces are a class of very well-behaved hyperbolic metric spaces. Their geometry is more precise, and many inequalities are in fact equalities, yielding better estimates in the asymptotic Schur orthogonality relations. This is essentially due to an equidistribution result of Roblin, deduced from the mixing property of the geodesic flow. The CAT(-1) setting is the perfect receptacle to develop the phenomenon of mixing of the geodesic flow as well as several of its consequences, and to generalize some deep results of ergodic theory of discrete groups in negative curvature \cite{Roblin2003}.  In this section we survey the geometry of CAT(-1) spaces. We freely rely on \cite{Bourdon1995} where the reader could consult for further details.

\subsection{CAT(-1) spaces and groups}\label{CATspaces}

Let \((X,d)\) be a proper (i.e.\ the balls are relatively compact) geodesic metric space. It is said to be CAT(-1) if every geodesic triangle in \(X\) is thinner than its comparison triangle in the hyperbolic plane, see \cite[Introduction]{Bridson1999}. This implies that \(X\) is hyperbolic, and since it is geodesic, the discussion of Section~\ref{sec:geometric-setting} applies to \(X\) with genuine geodesics used in place of rough geodesics. Moreover, if $a_n\to \xi\in\partial X$, and $b_n\to \eta\in\partial X$, then the corresponding Gromov products converge without the additive error, which appears in the general hyperbolic case, namely we have
\begin{equation}
  (\xi,\eta)_x=\lim_{n\to\infty}(a_n,b_n)_x.
\end{equation}
Finally, the formula 
\begin{equation}\label{distance}
        d_x(\xi,\eta)=e^{-(\xi,\eta)_x}
\end{equation}
already defines a metric on $\bd X$ (we set $d_x(\xi,\xi)=0$). This is due to M. Bourdon, we refer to \cite[2.5.1 Th\'eor\`eme]{Bourdon1995} for more details. 

Now, let $\Gamma$ be a non-elementary discrete group of isometries of $X$. The limit set of $\Gamma$, denoted by $\Lambda_{\Gamma}$ is the set of all accumulation points of a fixed orbit of \(\Gamma\) in $\bd X$, i.e.\ $\Lambda_{\Gamma}=\overline{\Gamma x}\cap \partial X$, where the closure is taken in $\overline{X}$. It is independent of the choice of $x\in X$. The geodesic hull $GH(\Lambda_{\Gamma})$ is defined as the union of all geodesics in $X$ with both endpoints in $\Lambda_{\Gamma}$, and the convex hull of $\Lambda_{\Gamma}$, denoted by $CH(\Lambda_{\Gamma})$, is the smallest subset of $X$ containing $GH(\Lambda_{\Gamma})$ with the property that every geodesic segment between any pair of points $x,y \in CH(\Lambda_{\Gamma})$ also lies in $CH(\Lambda_{\Gamma})$. We say that $\Gamma$ is \emph{convex cocompact} if it acts cocompactly on  $CH(\Lambda_{\Gamma})$. It is a well known fact that convex cocompact groups are hyperbolic, see \cite[Corollaire 1.8.3]{Bourdon1995}. In this case, \(\bd\Gamma = \Lambda_{\Gamma}\).

\subsection{Bowen-Margulis-Sullivan measures}\label{BMS}

We return to the general case of a discrete group of isometries of \(X\), possibly non-hyperbolic. It turns out that one can still perform a construction similar to the one yielding the Patterson-Sullivan measures. First, one defines the critical exponent $\alpha(\Gamma)$ of $\Gamma$ by
\begin{equation}
  \alpha(\Gamma)=\inf \Big\{s\in \mathbb{R}^{*}_{+} : \sum_{g \in \Gamma} e^{-sd(g x,x)} <\infty \Big\}
\end{equation}
(in case of a hyperbolic group, we have \(\alpha=D\epsilon = \log\omega\)). Now, a \emph{\(\Gamma\)-invariant conformal density of dimension \(\alpha\)} is a family of positive finite measures \(\{\mu_{x} : x\in X\}\) on \(\ov X\) such that
\begin{enumerate}
\item for \(g\in\Gamma\) we have \(g_{*}\mu_{x}=\mu_{gx}\),
\item for \(x,y\in X\) the measures \(\mu_{x}\) and \(\mu_{y}\) are equivalent, and moreover
  \begin{equation}
    \frac{d\mu_{y}}{d\mu_{x}}(\xi) = \omega^{\beta_{\xi}(x,y)},
  \end{equation}
  where \(\beta_{\xi}(x,y)\) is the \emph{horospherical distance of \(x\) and \(y\) relative to \(\xi\)}, defined as the limit
  \begin{equation}\label{horospherical}
    \beta_{\xi}(x,y) = \lim_{t\rightarrow \infty} d(x,\gamma(t)) - d(y,\gamma(t))
  \end{equation}
  where \(\gamma\) is any geodesic ray in \(X\) representing \(\xi\).
\end{enumerate}
In case of a discrete group \(\Gamma\) of isometries of a CAT(-1) space, there always exists a conformal density supported in the limit set \(\Lambda_{\Gamma}\). For a proof, see \cite{BurgerMozes1996} and \cite{Bourdon1995}.

In \cite{Sullivan1979}, D. Sullivan constructed measures on the unit tangent bundle of the $n$-dimensional real hyperbolic space, and proved some striking results for this new class of measures. We will now briefly recall the definitions of analogous measures in CAT(-1) spaces. 
We follow \cite[Chapitre 1C]{Roblin2003} where the reader could find more details.

Let $SX$ be the set of isometries from $\RR$ to $(X,d)$ endowed with the topology of uniform convergence on compact subsets of $\RR$. In other words, $SX$ is the set of geodesics of $X$ parametrized by $\RR$. We have a canonical projection \(\gamma\mapsto \gamma(0)\) from $SX$ to $X$, playing the role of the projection from the unit tangent bundle of a manifold to the manifold. Notice that in the setting of CAT(-1) spaces this map may be non-surjective, as the geodesics need not be bi-infinite.

The trivial flow on $\RR$ induces a continuous flow $(\Phi^{t})_{t\in \mathbb{R}}$ on $SX$, called the \emph{geodesic flow} , given by \(\Phi^{t}(\gamma)(s)=\gamma(s+t)\). For $\gamma\in SX$, we will denote by \(\gamma^{\pm}\in\bd X\) the endpoints of \(\gamma\), i.e.\ the limits \(\lim_{t\to\pm\infty} \gamma(t)\). Let
\begin{equation}
\partial^{2}X = \partial X \times \partial X \setminus \left\{ (x,x) : x\in \partial X \right\}.
\end{equation}
If we fix a basepoint \(x\in X\), we may identify $SX$ with $\bd^{2}X\times \RR$ via
\begin{equation}
\gamma \mapsto (\gamma^{-},\gamma^{+},\beta_{\gamma^{-}}(\gamma(0),x)). \end{equation}
Under this identification, called  the \emph{Hopf parametrization}, the geodesic flow can be written as
\begin{equation}
\Phi^{t}(\xi,\eta,s) = (\xi,\eta,s+t),  
\end{equation}
and it commutes with the action of $\Gamma$ on $\bd^{2}X\times \RR$ by
\begin{equation}
g \cdot(\xi,\eta,s )=(g \xi,g  \eta, s+\beta_{\xi}(x,g^{-1} x) ).
\end{equation}
  
Using a $\Gamma$-invariant conformal density of dimension $\alpha$, we define the \emph{Bowen-Margulis-Sullivan measure} \(m\) on \(SX\), often referred to as the \emph{BMS measure}, by
\begin{equation}
dm(\xi,\eta,s)=\frac{d\mu_{x}(\xi) d\mu_{x}(\eta) ds}{d_{x}(\xi,\eta)^{2\alpha}}\cdot
\end{equation}
It is invariant under both the geodesic flow \(\Phi^{t}\), and the action of \(\Gamma\). The latter action admits a measurable fundamental domain, which can be identified with the quotient \(SX/\Gamma\). We denote by $m_{\Gamma}$ the measure on $SX/\Gamma$ corresponding to the restriction of \(m\) to the fundamental domain, and say that $\Gamma$ \emph{admits a finite BMS measure} if $m_{\Gamma}$ is finite. 

\subsection{Equidistribution and asymptotic Schur orthogonality}
\label{sec:equid-asympt-orth-1}

The \emph{translation length} of an element $g \in \Gamma$ is defined as
\begin{equation}
t(g):=\inf_{x\in X} d(x,gx),
\end{equation}
and the \emph{spectrum} of $\Gamma$ is the subgroup of $\RR$ generated by $t(g)$ where $g$ ranges over the hyperbolic isometries in $\Gamma$. If it is a discrete subgroup, we say that $\Gamma$ has  \emph{arithmetic spectrum}. We are interested in discrete groups with a non-arithmetic spectrum, because they guarantee mixing properties of the induced geodesic flow on \(SX/\Gamma\). This condition is satisfied for isometry groups of Riemannian surfaces, hyperbolic spaces, and CAT(-1) spaces with a non-trivial connected component in their limit set. We refer to \cite{Dalbo1999} and to \cite[Proposition 1.6, Chapitre 1]{Roblin2003} for more details.
The proof of mixing can be found in \cite[Proposition 7.7]{Babillot2002} in the case of negatively curved manifolds, and in \cite[Chapitre 3]{Roblin2003} for general CAT(-1) spaces.  

Denote by $\delta_{x}$ the unit Dirac mass centered at $x$. Using the mixing property of the induced geodesic flow on \((SX/\Gamma,m_{\Gamma})\), in \cite[Th\'eor\`eme 4.1.1, Chapitre 4]{Roblin2003}, Roblin proved the following result.

\begin{theorem}[T. Roblin]\label{roblin}
  Let $\Gamma$ be a discrete group of isometries of a CAT(-1) space $X$ with a non-arithmetic spectrum. Assume that $\Gamma$ admits a finite BMS measure associated to a $\Gamma$-invariant conformal density $\mu$ of dimension $\alpha=\alpha(\Gamma)$. Then for all $x,y \in X$ we have
  \begin{equation}
    \lim_{n\to\infty} \alpha e^{-\alpha n} \lVert m_{\Gamma}\rVert \sum_{ \substack{g\in\Gamma \\ d(x,gy)<n}}\delta_{g^{-1}  x} \otimes \delta_{g  y} = \mu_{x} \otimes \mu_{y}
  \end{equation}
  in the weak* topology of $C(\overline{X} \times \overline{X})^{*}$.
\end{theorem}

 Let $A_{R}(x,y)=\{  g \in \Gamma | R- h\leq d(x,g y) \leq R+h\}$ for some sufficiently large $h>0$. We now have the following.
 
\begin{corollary}[T. Roblin]\label{coro:roblin-equi}
  Under the assumptions of Theorem~\ref{roblin}, for all $x,y \in X$ we have
  \begin{equation}
    \lim_{R\to\infty}\frac{1}{\card{A_{R}(x,y)}}\sum_{ g \in A_{R}(x,y) }\delta_{g^{-1}x} \otimes \delta_{gy} = \mu_{x} \otimes \mu_{y} 
  \end{equation}
  in  the weak* topology on $C(\overline{X} \times \overline{X})^{*}$.
\end{corollary}

Now, suppose again that \(\Gamma\) is convex cocompact, and thus hyperbolic. Fix a basepoint \(o\in X\) and denote \(d_{\Gamma}(g,h) = d(go,ho)\). If the orbit map \(g\mapsto go\) is not injective, this is only a pseudometric, which can be remedied by applying a bounded perturbation
\begin{equation}
  d_{\Gamma}^{+}(g,h) =
  \begin{cases}
    0 &\text{if \(g=h\),}\\
    d_{\Gamma}(g,h) + 1 &\text{otherwise,}
  \end{cases}
\end{equation}
yielding a roughly isometric metric \(d_{\Gamma}^{+}\in\mathcal{D}(\Gamma)\). Recall that we may identify \(\bd\Gamma\) with \(\Lambda_{\Gamma}\), and then the measure \(\mu_{o}\) of the \(\Gamma\)-invariant conformal density is then in the Patterson-Sullivan class corresponding to \(d_{\Gamma}^{+}\). If \(A_{R,h}\) is the annulus corresponding to \(d_{\Gamma}^{+}\), then by Corollary \ref{coro:roblin-equi}, for sufficiently large \(h>0\) we have
\begin{equation}
  \lim_{R\to\infty}\frac{1}{\card{A_{R,h}}}\sum_{ g \in A_{R,h} }\delta_{go} \otimes \delta_{g^{-1}o} = \mu_{o} \otimes \mu_{o},
\end{equation}
and since the orbit map \(g\mapsto go\) extends to a continuous map \(\ov\Gamma\to \ov X\) identifying \(\bd\Gamma\) with \(\Lambda_{\Gamma}\), we also have
\begin{equation}
  \lim_{R\to\infty}\frac{1}{\card{A_{R,h}}}\sum_{ g \in A_{R,h} }\delta_{g} \otimes \delta_{g^{-1}} = \mu_{o} \otimes \mu_{o},
\end{equation}
obtaining a family of measures on \(\Gamma\) yielding a sharper version of the Equidistribution Theorem~\ref{prop:equidistribution}.

From this, we deduce the asymptotic Schur orthogonality relations in CAT(-1) spaces. Let \(\pi\) be the boundary representation associated to the metric \(d_{\Gamma}^{+}\), and recall that
\begin{equation}
  \tilde{\pi}(g) = \frac{\pi(g)}{\Xi(g)}
\end{equation}
is the normalization with the Harish-Chandra function.

\begin{theorem}[Asymptotic Schur Orthogonality] \label{prop:asymptotic-orthogonality-conv}
  Assume that $\Gamma$ is a convex cocompact group of isometries of a CAT(-1) space with a non-arithmetic spectrum. Then for sufficiently large \(h>0\), all continuous functions \(f_{1},f_{2} \in C(\ov{X})\), and vectors \(v_{1}, v_{2}, w_{1}, w_{2}\in L^{2}(\bd X,\mu_{o})\)
  we have
  \begin{multline}
    \lim_{R\to\infty} \frac{1}{\card{A_{R,h}}} \sum _{g\in A_{R,h} }f_{1}(g o)
    f_{2}(g^{-1} o) \langle \tilde{\pi}(g)v_{1}, w_{1}\rangle
    \conj{\langle \tilde{\pi}(g)v_{2}, w_{2}\rangle} = \\ 
    = \langle f_{2}|_{\bd X} v_{1},v_{2}\rangle
    \conj{\langle w_{1},f_{1}|_{\bd X} w_{2}\rangle}.
  \end{multline}
\end{theorem}

Note that this theorem implies Theorem A of \cite{Boyer2014}, an ergodic theorem \`a la Bader-Muchnik.

\begin{remark}
We may wonder whether this convergence holds for non-uniform lattices in $SL(2,\mathbb{R})$ (or more generally in rank one semisimple Lie groups) for boundary representations associated to the Lebesgue measure on the boundary. It turns out that, since the Lebesgue measure satisfies Ahlfors regularity (with respect to the visual metric associated to the Riemannian metric of the symmetric space) the convergence holds for Lipschitz functions on the boundary. But it is important to note that the convergence does not extend to $L^{2}$ functions because it requires property RD with respect to the Riemannian metric. 
\end{remark}

\begin{remark}
  Let \(\Gamma\) be a non-abelian free group acting on its Cayley
  graph \(X\) corresponding to the standard generators. Although in
  this case the spectrum is arithmetic, by \cite{BoyerLobos2016} we still have an
  equiuistribution theorem, leading to asymptotic orthogonality
  relations
  \begin{multline}
    \lim_{n\to\infty} \frac{1}{\card{S_{n}}} \sum_{g\in S_{n}} f_{1}(g)f_{2}(g^{-1}) \langle \tilde{\pi}(g)v_{1}, w_{1}\rangle
    \conj{\langle \tilde{\pi}(g)v_{2}, w_{2}\rangle}=\\= \langle f_{2}|_{\bd \Gamma} v_{1},v_{2}\rangle
    \conj{\langle w_{1},f_{1}|_{\bd \Gamma} w_{2}\rangle},
  \end{multline}
  where \(f_{1},f_{2}\in C(\Gamma \cup \bd\Gamma)\), and \(S_{n}=\{ g\in \Gamma : \abs{g}=n\}\).
\end{remark}

\section{Simply connected manifolds and Gibbs measures}\label{sec:simply-conn-manif}

\subsection{The Gibbs stream}
\label{sec:gibbs-stream}

We briefly recall the theory of Gibbs measures in negative curvature, following \cite{Paulin2015} where the reader could consult for more details. Let $X$ be a complete simply connected Riemannian manifold of dimension at least $2$, and pinched sectional curvature $-b^{2}\leq K\leq -1$ with $b\geq 1$, equipped with its Riemannian distance denoted by $d$, and acted upon by a non-elementary discrete group of isometries \(\Gamma\). Such manifolds form a particular class of CAT(-1) spaces in which the Patterson-Sullivan measures can be further generalized. We will assume that \(X=\tilde{M}\) is the universal cover of a manifold \(M\) with fundamental group \(\Gamma\), and denote \(q\colon X\to M\) the covering map.

Let $p\colon T^{1}X\rightarrow X$ be the unit tangent bundle of $X$, and for \(v\in T^{1}X\) denote by \(\gamma_{v}\) the unique geodesic in \(X\) with \(\gamma_{v}'(0)=v\). We may equip \(T^{1}X\) with the following metric
\begin{equation}
  d_{T^{1}X}(v,w) = \frac{1}{\sqrt{\pi}} \int_{-\infty}^{\infty} d( \gamma_{v}(t),\gamma_{w}(t)) e^{-t^{2}/2}\,dt.
\end{equation}  
Now, for a H\"older-continuous map  $F \colon T^{1}M \rightarrow \mathbb{R}$, called a \emph{potential}, let $\widetilde{F} = F \circ q$ be its lift to a $\Gamma$-invariant potential on $T^{1}X$. We will require that \emph{$\widetilde{F}$ is symmetric}, i.e.\ invariant under the antipodal map \(v\mapsto -v\).
 
For all  $x,y \in X$,  define
\begin{equation}
  d^{F}(x,y)= 
  \int_{0}^{d(x,y)} \widetilde{F}(\gamma'(t))\,dt 
\end{equation}
where \(\gamma\) is the geodesic segment from $x$ to $y$. A priori,
$d^{F}$ is not non-negative and is far from being a metric,
nevertheless the symmetry of $\widetilde{F}$ implies that
\begin{equation} \label{sym}
d^{F}(x,y)=d^{F}(y,x).
\end{equation}
Now, for \(\xi\in\bd X\) define the Gibbs cocycle as
\begin{equation}\label{gibbs}
C^{F}_{\xi}(x,y)=
\lim_{t\rightarrow +\infty} d^{F}(y,\xi_{t})-d^{F}(x,\xi_{t}),
\end{equation}
where $\xi_{t}$ is any geodesic ray representing $\xi$.
Observe that if $\widetilde{F}=-1$ the Gibbs cocycle is nothing but the horospherical distance \(\beta_{\xi}(x,y)\).

Let \(\sigma\in\RR\). A family of positive finite measures \(\{ \mu_{x}^{F}: x\in X\}\) on \(\ov X\) is a \emph{Gibbs stream of dimension $\sigma$ for $(\Gamma,F)$} if 
\begin{enumerate}
\item For all $x$ and $y$ in $X$, $\mu_{x}^F$ and $\mu^{F}_{y}$ are
  equivalent, and 
  \begin{equation}
    \frac{ d\mu^{F}_{y}}{d\mu^{F}_{x}}(\xi)=e^{C^{F-\sigma}_{\xi}(x,y)},
  \end{equation}

\item For all $g \in \Gamma$, and $x \in X$ we have
  $g_{*}\mu^{F}_{x}=\mu^{F}_{g x}$.
\end{enumerate}
In this context, the critical exponent of $(\Gamma,F)$ is defined as
\begin{equation}
\sigma_{\Gamma,F}=\limsup_{n\to\infty} \frac{1}{n}\sum_{g\in S_{n,h}}e^{d^{F}(x,g x)},
\end{equation}
where \(S_{n,h}=\{ g\in\Gamma : n-h\leq d(g x,x)\leq n \}\), and \(h>0\) is sufficiently large.

\begin{proposition}(S-J. Patterson, \cite{Patterson1976}) If $\sigma_{\Gamma,F}<\infty$,
  then there exists at least one Gibbs stream of dimension
  $\sigma_{\Gamma,F}$ with support exactly equal to
  $\Lambda_{\Gamma}$.
\end{proposition}

\subsection{Equidistribution and asymptotic Schur orthogonality}\label{Gibbs}

For all $x \in X$ and $\xi,\eta\in \bd X$ define the \emph{gap map} as
\begin{equation}
  D_{x,F}(\xi,\eta) = \exp \lim_{t\to\infty} \frac{ d^{F}(x,\xi_{t}) + d^{F}(x,\eta_{t}) - d^{F}(\xi_{t},\eta_{t})}{2} .
\end{equation}
When \(F=-1\), the gap map $D_{x,F}$ is just the visual metric $d_{x}$ from Section~\ref{CATspaces}. Also, note the $\Gamma$-invariance of \(D_{x,F}\), namely
\begin{equation}
  D_{g x, F}(g\xi,g\eta)=D_{x,F}(\xi,\eta)
\end{equation}
for all $g \in \Gamma$ and $\xi,\eta\in \bd X$.

Let $\mu^{F}$ be a Gibbs stream for $(\Gamma,F)$ of dimension $\sigma=\sigma_{(\Gamma,F)}$. Once we have fixed a base point $x\in X$ and used the Hopf parametrization, define on \(T^{1}X\) \emph{the Gibbs measures  associated with $\mu^{F}$} by   
\begin{equation}\label{Gibbsmeasure}
dm(\xi,\eta,t)=\frac{d\mu^{F}_{x }(\eta)  d\mu^{F}_{x }(\xi) dt }{D_{x,F-\sigma}(\eta,\xi)^{2}}.
\end{equation}
Similarly as in Section~\ref{BMS}, this measure descends to a measure \(m^{F}\) on the quotient \(T^{1}X/\Gamma=T^{1}M\), called \emph{the Gibbs measure on \(T^{1}M\) associated with \(\mu^{F}\)}. It is invariant under the geodesic flow, and in the case of convex cocompact groups, always finite.

We are now ready to state the equidistribution theorem for the Gibbs stream due to Paulin, Policott and Shapira \cite{Paulin2015}.
 
\begin{theorem}[Paulin-Pollicott-Schapira]\label{theo:roblin-pps}
  Assume that $\sigma$ is finite and positive. If $m^{F}$ is finite and mixing under the geodesic flow on $T^{1}X$, then for all $x,y \in X$ and sufficiently large $h>0$
  \begin{equation}
    \lim_{n\to\infty} \frac{\sigma \lVert m^{F}\rVert e^{-\sigma n}}{1 - e^{-c\sigma}} \sum_{g\in S_{n,h} } e^{d^{F}(x,g y)}\delta_{g^{-1}  x} \otimes \delta_{g  y} = \mu^{F}_{x} \otimes \mu^{F}_{y}
\end{equation}
in the weak* topology of $C(\overline{X} \times \overline{X})^{*}$.
\end{theorem}

In \cite{BoyerMayeda2016}, the irreducibility and classification of the quasi-regular representations associated with Gibbs streams are proved. In particular, the Harish-Chandra functions have been studied, and some useful estimates were obtained.

Pick a basepoint $x\in X$ and let $\mu=\mu^{F}_{x}$ be the measure in the Gibbs stream associated to $x$. Denote by \(\pi\) the corresponding quasi-regular representation of \(\Gamma\) on  \(L^{2}(X,\mu)\), and consider its Harish-Chandra function $\Xi(g)=\langle\pi(g)\one,\one\rangle$. Again, let \(\tilde{\pi}(g)=\pi(g)/\Xi(g)\). We have

\begin{proposition}\label{prop:harish-chandra-gibbs}
  Assume that $\Gamma$ is convex cocompact and that $F$ is symmetric. Then we have
  \begin{equation}
    \Xi(g)\asymp e^{d^{F-\sigma}(x,gx) /2} (1 + d(x,gx)) .
  \end{equation}
\end{proposition}

In the proof of Theorem \ref{prop:asymptotic-orthogonality-general}, a crucial role is played by  \cite[Lemma 5.3]{Garncarek2014}, asserting that for \(v,w \in \Lip(\bd\Gamma,d_{\epsilon})\) we have
\begin{equation}
  \abs{ \langle \tilde{\pi}_{\mu_{PS}}(g)v,w\rangle - v(\check{g})\conj{w(\hat{g})} } \lasymp_{v,w} (1+\abs{g})^{-1/D}.
\end{equation}
In the context of Gibbs stream we have the following analogous result.
\begin{lemma}\label{lemma:estim-coeff}
  Assume that $\Gamma$ is convex cocompact and that $F$ is symmetric. Then for any $\epsilon>0$ there exists $C_{\epsilon}>0$ such that for all $v,w\in {\rm Lip}(\partial X, d_{x})$
  \begin{equation}
    \abs{ \langle \tilde{\pi}(g)v,w\rangle - v(\check{g})\conj{w(\hat{g})} } \lasymp_{v,w} \epsilon+\frac{C_{\epsilon}}{d(x,gx)}
  \end{equation}
\end{lemma}

\begin{proof}
  Following the proof of \cite[Lemma 5.3]{Garncarek2014} we have
  \begin{equation}
    \begin{split}
      \abs{ \langle \tilde{\pi}(g)v,w \rangle - v(\check{g})
        \conj{w(\hat{g})} } & \leq \norm{v}_{\infty} \int_{\partial X}
      \frac{e^{C_{\xi}^{F-\sigma}(x,gx)/2}}{\Xi(g)} \abs{w(\xi) -
        w(\hat{g})} \,d\mu(\xi) \\ 
      & + \norm{w}_{\infty} \int_{\partial X}
      \frac{e^{C_{\xi}^{F-\sigma}(x,g^{-1}x)/2}}{\Xi(g^{-1})}
      |v(\xi)-v(\check{g})| \,d\mu(\xi).
    \end{split}
  \end{equation}
  Both terms are similar, so it is enough to study only the first one. Take \(\epsilon>0\) and define $B=B(\hat{g},\epsilon)$. We will estimate the integrals over \(B\) and \(\bd X\setminus B\) separately.

  Since $w$ is in $\Lip(\bd X, d_{o})$, we have
  \begin{equation}
    \int_{B}  \frac{e^{C_{\xi}^{F-\sigma}(x,gx)/2}}{\Xi(g)} |w(\xi)-w(\hat{g})| \,d\mu(\xi)\lasymp_{w} \epsilon .
  \end{equation}
  The integral over \(\bd X\setminus B\) is controlled thanks to the estimate \cite[Lemma 5.1]{BoyerMayeda2016}. More precisely we have  
  \begin{equation}
  \int_{\partial X\backslash B}  \frac{e^{C_{\xi}^{F-\sigma}(x,gx)}}{\Xi(g)} |w(\xi)-w(\hat{g})| \, d \mu(\xi) \lasymp_{w} \frac{\langle \pi(g)\one, \one_{\partial X\backslash B}\rangle}{\Xi(g)}
  \leq \frac{C_{\epsilon}}{d(x,gx)},
  \end{equation}
  where the last inequality follows from  \cite[Lemma 5.1]{BoyerMayeda2016} and complete the proof.
\end{proof}

We now obtain the following asymptotic Schur orthogonality relations for quasi\-regular representations associated with Gibbs streams. Again, the assumption of non-arithmeticity guarantees the mixing property of the geodesic flow in the context of Gibbs measures, see \cite[Theorem 8.1]{Paulin2015}. By taking $f_{2}=\one_{\ov{X}}$ and $v_{2}=w_{2}=\one_{\bd{X}}$ we recover Theorem A in \cite{BoyerMayeda2016}, another ergodic theorem \`a la Bader-Muchnik, used to prove irreducibility of the boundary representation. 
\begin{theorem}
Assume that $\Gamma$ is convex cocompact with a non-arithmetic spectrum, and that the potential $F$ is symmetric. Then for sufficintly large \(h>0\), all $f_{1},f_{2}\in C(\ov X)$ and all $v_{1},v_{2},w_{1},w_{2} \in L^{2}(\bd X, \mu)$ we have
\begin{multline}
\lim _{R \to \infty}C  e^{-\sigma R}\sum_{ g\in S_{R,h}} e^{d^{F}(x,gx)} f_{1}(g) f_{2}(g^{-1}) \langle \tilde{\pi}(g)v_{1}, w_{1}\rangle \conj{\langle \tilde{\pi}(g)v_{2}, w_{2}\rangle}= \\
 = \langle f_{2}|_{\bd X} v_{1},v_{2}\rangle
      \conj{\langle w_{1},f_{1}|_{\bd X} w_{2}\rangle},
\end{multline}
where
\begin{equation}
  C = \frac{\sigma \norm{m_{F}}}{1-e^{-h\sigma}}.
\end{equation}
\end{theorem}

The proof is a straightforward adaptation of the proof of the Asymptotic Orthogonality Theorem~\ref{prop:asymptotic-orthogonality-general}. To show the uniform boundedness part, one just needs to observe that for \(g\in S_{R,h}\)
\begin{equation}
  \frac{Ce^{-\sigma R} e^{d^{F}(x,gx)}}{\Xi(g)^{2}} \asymp_{h} \frac{ C e^{ d^{F}(x,gx)-\sigma d(x,gx)}}{\Xi(g)^{2}} = \frac{ C e^{ d^{F-\sigma}(x,gx)}}{\Xi(g)^{2}}\asymp \frac{1}{(1+R)^{2}}.
\end{equation}
For the convergence part, one uses Lemma~\ref{lemma:estim-coeff}.

\section{Application to monotony}
\label{sec:application-monotony}

Here we will present an example of application of our results. We will deal with the phenomenon of \emph{monotony} of representations, described by Kuhn and Steger in \cite{Kuhn2001}. 

Let \(\Gamma\) be a non-abelian free group. By a \emph{generalized boundary representation} (called \emph{boundary representation} in \cite{Kuhn2001}) of \(\Gamma\) on a Hilbert space \(\mathcal{H}\) we will understand a pair \((\sigma,\rho)\), where
\begin{enumerate}
\item \(\sigma\) is a unitary representation of \(\Gamma\) on
  \(\mathcal{H}\),
\item  \(\rho\) is a representation of the \(C^{*}\)-algebra \(C(\bd\Gamma)\) on \(\mathcal{H}\),
\item \(\sigma\) and \(\rho\) satisfy the condition \(\sigma(g)\rho(f)\sigma(g^{-1}) = \rho(f\circ g^{-1})\) for all \(g\in\Gamma\) and \(f\in C(\bd\Gamma)\).
\end{enumerate}
In other words, this is a representation of the crossed product $C(\bd \Gamma)\rtimes \Gamma$, where $\Gamma$ acts on $C(\bd \Gamma)$ via $g \cdot f=f \circ g^{-1}$.

Now, let \(\pi\) be a unitary representation of \(\Gamma\) on a Hilbert space \(\mathcal{H}\). A \emph{boundary realization} of \(\pi\) on a Hilbert space \(\mathcal{H}'\) is a triple \((\iota,\sigma,\rho) \), where \((\sigma,\rho)\) is a generalized boundary representation of \(\Gamma\) on \(\mathcal{H}'\), and \(\iota\colon \mathcal{H}\to\mathcal{H}'\) is a \(\Gamma\)-equivariant isometric embedding, such that \(\rho(C(\bd\Gamma))\iota(\mathcal{H})\) is dense in \(\mathcal{H}'\). If \(\iota\) is bijective, the realization is said to be \emph{perfect}.

Boundary representations studied herein come equipped with tautological perfect realizations. Indeed, if \(\pi\) is a boundary representation on \(\mathcal{H}=L^{2}(\bd\Gamma,\mu)\), then the pair \((\pi,m)\), where \(m\) is the multiplication action \(m(f)\phi=f\phi\) of \(C(\bd\Gamma)\) on \(\mathcal{H}\), is a generalized boundary representation of \(\Gamma\), and together with the identity map of \(\mathcal{H}\) it yields a perfect boundary realziation of \(\pi\).

The fundamental idea of Kuhn and Steger in \cite{Kuhn2001} is the use of boundary realizations to classify  representations of the free group weakly contained in the regular representation. They propose three classes of representations, \emph{monotonous, duplicitous, and odd}, distinguished by the behavior of their boundary realizations, and inspired by the case where \(\Gamma\) is a lattice in \(SL_{2}(\RR)\).

Let \(\pi\) be a unitary representation of \(\Gamma\), and let  $(\iota,\sigma,\rho)$ and $(\iota',\sigma',\rho')$ be two boundary realizations of $\pi$, on Hilbert spaces \(\mathcal{H}\), and \(\mathcal{H}'\), respectively. We will consider them \emph{equivalent}, if there exists a unitary isomorphism $U \colon \mathcal{H} \to \mathcal{H}'$ intertwining \((\sigma,\rho)\) with \((\sigma',\rho')\), such that \(U\iota=\iota'\). 

We say that the representation \(\pi\) satisfies \emph{monotony}, if it admits a unique (up to equivalence) boundary realization, and moreover this realization is perfect. If \(\Gamma\) is a lattice in \(SL_{2}(\RR)\), then the restrictions of the endpoint  spherical principal  series representations to \(\Gamma\) are monotonous.

The representation \(\pi\) is said to satisfy \emph{duplicity}, if it admits exactly two non-equivalent perfect boundary realizations $(\iota_{1},\sigma_{1},\rho_{1})$ and $(\iota_{2},\sigma_{2},\rho_{2})$, and moreover any imperfect realization is, up to equivalence, a combination of these two, given by the embedding \(\iota(v)=t_{1}\iota_{1}(v)\oplus t_{2}\iota_{2}(v)\), where \(t_{1},t_{2}>0\), and \(t_{1}^{2}+t_{2}^{2}=1\). If \(\Gamma\) is a lattice in \(SL_{2}(\RR)\), this happens for the restrictions of non-endpoint principal series representations.

Finally, a representation \(\pi\) of \(\Gamma\) on a Hilbert space \(\mathcal{H}\) satisfies \emph{oddity}, if it admits a unique boundary realization \((\iota,\sigma,\rho)\) on \(\mathcal{H}'\), this realization is imperfect, and moreover if \(\mathcal{K}\) is the orthogonal complement of \(\iota(\mathcal{H})\) in \(\mathcal{H}'\), then the restriction \(\rho\res_{\mathcal{K}}\) is irreducible, not equivalent to \(\pi\), and \((\id_{\mathcal{K}},\rho,\sigma)\) is its unique boundary realization. In case where \(\Gamma\) is a lattice in \(SL_{2}(\RR)\), this situation happens for the restrictions of the mock discrete series representations.

\begin{conjecture}[Kuhn-Steger]
Any irreducible unitary representation of $\Gamma$ weakly contained
in the regular representation is monotonous, or duplicitous, or odd.
\end{conjecture}

Examples of representations of non-abelian free groups satisfying monotony were defined in \cite{Kuhn1996}, and the proof of their monotony can be found in \cite{Kuhn2001}. The authors mention that ``monontony is a generic condition''.
 
\begin{remark} Monotony of a boundary representation can be thought of as a \emph{rigidity} phenomenon. Indeed, given a quasi-regular representation $\pi_{\mu}$ associated with a measure $\mu$ on the boundary $\bd \Gamma$ of the free group, there is a unique (up to equivalence) way  for the algebra of continuous functions $C(\partial \Gamma)$ to act on $L^{2}(\partial \Gamma,\mu)$ while satisfying the cross product relation with $\pi_{\mu}$, namely the multiplication action. In other words, there is an unique way to extend the unitary representation $\pi_{\mu}$ of $\Gamma$ to a representation of the cross product $C(\partial \Gamma)\rtimes \Gamma$.
\end{remark}

Whereas we do not investigate in this paper the phenomenon of duplicity or oddity, one of the consequences of our Asymptotic Orthogonality Theorem is a new strategy for proving monotony of a wider class boundary representations. It allows us to obtain the following result, strengthening the speculation of Kuhn and Steger about monotony as a generic condition.

\begin{theorem}\label{prop:free-group-monotony}
  Let \(\Gamma\) be a non-abelian free group, and let $\mu$ be the Patterson-Sullivan measure on \(\bd\Gamma\) associated with a metric $d\in\mathcal{D}(\Gamma)$. Then the associated boundary representation $\pi$ satisifes monotony.
\end{theorem}

Our method will be still relying on the \emph{Good Vector Bound (GVB)} introduced by Kuhn and Steger. Endow \(\Gamma\) with its word-length \(d_{w}\) associated to its free basis, and denote by \(S_{n}\subseteq\Gamma\) the set of elements of length \(n\). A representation \(\pi\) of \(\Gamma\) on a Hilbert space \(\mathcal{H}\) satisfies GVB if there exists a non-zero vector \(w\in\mathcal{H}\) such that for all \(v\in\mathcal{H}\) we have
\begin{equation}\label{eq:GVB}
  \sum_{g \in S_{n}}\abs{\langle \pi(g)v,w \rangle}^{2}\leq C\norm{v}^{2}_{2}.
\end{equation}
The key result is the following.

\begin{proposition}[{\cite[Proposition 2.7 and Corollary 2.8]{Kuhn2001}}]\label{GVBprop}
If a unitary representation of $\Gamma$ does not satisfy the Good Vector Bound then it satisfies monotony.
\end{proposition}

\begin{proof}[Proof of Theorem~\ref{prop:free-group-monotony}]
  We shall prove that $\pi$ cannot satisfy GVB. Suppose to the contrary that \eqref{eq:GVB} holds for some nonzero \(w\in\mathcal{H}\). Since the metrics \(d\) and \(d_{w}\) are quasi-isometric, for all \(\epsilon,R>0\), and $v\in \mathcal{H}$ we obtain
  \begin{equation}\label{eq:appl-gvb}
    \begin{split}
      \sum_{g\in A_{R,h}}
      \frac{\abs{\langle\pi(g)v,w\rangle}^{2}}{(1+R)^{1+\epsilon}}
      & \lasymp_{h}  \sum_{g\in\Gamma} \frac{\abs{\langle \pi(g)
          v,w\rangle}^{2} }{(1+d(g))^{1+\epsilon}} \lasymp \\
      & \lasymp
      \sum_{g\in\Gamma} \frac{\abs{\langle \pi(g) v,w\rangle}^{2}}
      {(1+d_{w}(g))^{1+\epsilon}}\lasymp_{\epsilon,w} \norm{v}^{2}.
    \end{split}
  \end{equation}
  On the other hand, for sufficiently large \(h>0\), and a suitable family of measures \(\{\mu_{R}\}\), by the Asymptotic Orthogonality Theorem~\ref{prop:asymptotic-orthogonality-general} we have
  \begin{equation}
    \norm{v}^{2}\norm{w}^{2} = \lim_{R\to\infty} \int \abs{\langle \tilde{\pi}(g)v,w \rangle}^{2}\,d\mu_{R} \lasymp_{h} \liminf_{R\to\infty} \sum_{g\in A_{R,h}} \frac{\abs{\langle\pi(g)v,w\rangle}^{2}}{(1+R)^{2}},
  \end{equation}
  so for large \(R\)
  \begin{equation}
    \sum_{g\in A_{R,h}} \frac{\abs{\langle\pi(g)v,w\rangle}^{2}}{(1+R)^{1+\epsilon}} \rasymp_{h} (1+R)^{1-\epsilon}\norm{v}^{2}\norm{w}^{2}.
  \end{equation}
  If we choose \(\epsilon<1\), this contradicts \eqref{eq:appl-gvb}.
\end{proof}

Thanks to the work of \cite{Blachere2011}, Theorem~\ref{prop:free-group-monotony} covers also quasi-regular representations associated with harmonic measures of random walks. Fix a word metric on \(\Gamma\) corresponding to the free basis, and let \(\nu\) be a symmetric random walk on \(\Gamma\). We say that \(\nu\) has \emph{exponential moment}, if for some \(\lambda>0\) we have
\begin{equation}
  \sum_{g\in\Gamma} e^{\lambda\abs{g}} \nu(g) < \infty.
\end{equation}
Let \(F(g,h)\) denote the probability that a trajectory of the random walk starting at \(g\) passes through \(h\). We will assume that for any \(r>0\) there exists \(C(r)\) such that
\begin{equation}\label{eq:walk-condition}
  F(g,h) \leq C(r) F(g,k)F(k,h)
\end{equation}
for all \(k\) within distance \(r\) from a geodesic joining \(g\) and \(h\) in the Cayley graph of \(\Gamma\). Both this condition, and having exponential moment are satisfied by any finitely supported \(\nu\), and condition~\eqref{eq:walk-condition} is due to Ancona \cite{Ancona1990}. Together they guarantee that the \emph{Green metric}
\begin{equation}
  d_{\nu}(g,h) = -\log F(g,h)
\end{equation}
is in the class \(\mathcal{D}(\Gamma)\). By \cite[Theorems 1.1(ii) and 1.5]{Blachere2011}, the harmonic measure \(\mu\) on \(\bd\Gamma\) associated to the random walk \(\nu\) is equivalent to the Patterson-Sullivan measure corresponding to the Green metric \(d_{\nu}\). Hence the following.

\begin{corollary}
  Let \(\Gamma\) be a non-abelian free group, and let $\nu$ be a symmetric random walk on \(\Gamma\) with an exponential moment, satisfying condition~\eqref{eq:walk-condition}. If $\mu$ is the associated harmonic measure on \(\bd\Gamma\), then the corresponding quasi-regular representation \(\pi_{\mu}\) satisfies monotony.
\end{corollary}


\bibliographystyle{plain}
\bibliography{monotony}

\end{document}